\documentclass{amsart}[12pt]
\usepackage{amsmath, amsthm, amscd, amsfonts,}
\usepackage{graphicx,float}

\usepackage{amssymb}

\usepackage{pb-diagram}

\DeclareMathOperator{\len}{l}

\DeclareMathOperator{\dom}{dom}
\DeclareMathOperator{\range}{range}

\DeclareMathOperator{\ran}{ran}

\def\MPB{{\mathbb{P}}}
\def\MQB{{\mathbb{Q}}}
\def\MRB{{\mathbb{R}}}
\def\MCB{{\mathbb{C}}}
\def\ds{\displaystyle}
\def\k{\kappa}
\def\sse{\subseteq}
\def\ss{\subset}
\def\l{\lambda}
\def\lan{\langle}
\def\ran{\rangle}
\def\a{\alpha}
\def\om{\omega}
\def\lora{\longrightarrow}
\def\ov{\overline}
\def\g{\gamma}
\def\d{\delta}
\def\b{\beta}

\newtheorem{theorem}{Theorem}[section]
\newtheorem{lemma}[theorem]{Lemma}

\newtheorem{definition}[theorem]{Definition}

\newtheorem{remark}[theorem]{Remark}
\newtheorem{notation}[theorem]{Notation}
\newtheorem{claim}[theorem]{Claim}
\numberwithin{equation}{section}

\def\MPB{{\mathbb{P}}}
\def\MQB{{\mathbb{Q}}}
\def\MRB{{\mathbb{R}}}
\def\MCB{{\mathbb{C}}}
\def\ds{\displaystyle}
\def\k{\kappa}
\def\sse{\subseteq}
\def\ss{\subset}
\def\l{\lambda}
\def\lan{\langle}
\def\ran{\rangle}
\def\a{\alpha}
\def\om{\omega}
\def\lora{\longrightarrow}
\def\ov{\overline}
\def\g{\gamma}
\def\d{\delta}
\def\b{\beta}

\def\l{\lambda}
\def\ds{\displaystyle}
\def\sse{\subseteq}

\def\lora{\longrightarrow}
\def\lan{\langle}
\def\ran{\rangle}
\def\ov{\bar}
\def\rmark{\mbox{$\rm\bf\rule{0.06em}{1.45ex}\kern-0.05em R$}}
\def\pmark{\mbox{$\rm\bf\rule{0.06em}{1.45ex}\kern-0.05em P$}}
\def\nmark{\mbox{$\rm\bf\rule{0.06em}{1.45ex}\kern-0.05em N$}}
\def\vdash{\mbox{$\rm\| \kern-0.13em -$}}
\newcommand{\lusim}[1]{\smash{\underset{\raisebox{1.2pt}[0cm][0cm]{$\sim$}}
{{#1}}}}

\usepackage{pb-diagram}

\def\l{\lambda}
\def\ds{\displaystyle}
\def\sse{\subseteq}

\def\lora{\longrightarrow}
\def\lan{\langle}
\def\ran{\rangle}
\def\ov{\bar}
\def\rmark{\mbox{$\rm\bf\rule{0.06em}{1.45ex}\kern-0.05em R$}}
\def\pmark{\mbox{$\rm\bf\rule{0.06em}{1.45ex}\kern-0.05em P$}}
\def\nmark{\mbox{$\rm\bf\rule{0.06em}{1.45ex}\kern-0.05em N$}}
\def\vdash{\mbox{$\rm\| \kern-0.13em -$}}

\begin{document}

\title[Adding a lot of Cohen reals by adding a few I]{Adding a lot of Cohen reals by adding a few I }

\author[M. Gitik and M. Golshani]{Moti Gitik and Mohammad Golshani}

\thanks{} \maketitle




\begin{abstract}
In this paper we produce models $V_1\sse V_2$ of set theory such
that adding $\k$-many Cohen reals to $V_2$ adds $\l$-many Cohen
reals to $V_1$, for some $\l>\k$. We deal mainly with the case
when $V_1$ and $V_2$ have the same cardinals.
\end{abstract}

\maketitle


\section{Introduction}

A basic fact about Cohen reals is that adding $\l$-many Cohen
reals cannot produce more than $\l$-many of Cohen reals \footnote{By ``$\lambda$-many Cohen reals'' we mean ``a generic object $\langle s_{\alpha}: \alpha<\lambda   \rangle$ for the poset $\MCB(\lambda)$ of finite partial functions from $\lambda\times\omega$ to $2$''.}. More
precisely, if $\lan s_\a:\a<\l\ran$ are $\l$-many Cohen reals over
$V$, then in $V[\lan s_\a:\a<\l\ran]$ there are no $\l^+$-many
Cohen reals over $V$. But if instead of dealing with one universe
$V$ we consider two, then the above may no longer be true.

The purpose of this paper is to produce models $V_1\sse V_2$ such
that adding $\k$-many Cohen reals to $V_2$ adds $\l$-many Cohen
reals to $V_1$, for some $\l>\k$. We deal mainly with the case
when $V_1$ and $V_2$ have the same cardinals.

\section{Models with the same reals}

In this section we produce models $V_1\sse V_2$ as above with the
same reals. We first state a general result.

\begin{theorem}
Let $V_1$
be an extension of $V$. Suppose that in $V_1:$

$(a)$ $\kappa < \lambda$ are infinite cardinals,

$(b)$ $\lambda$ is regular,

$(c)$ there exists an increasing sequence $\langle \kappa_n: n < \omega  \rangle$ of regular cardinals cofinal in $\kappa$; in particular $cf(\kappa) = \omega,$

$(d)$ there exists an increasing (mod finite) sequence $\langle
f_{\alpha}: \alpha < \lambda \rangle$ of functions in
$\ds\prod_{n<\omega}(\kappa_{n+1}\setminus\kappa_n),$\footnote{Note that condition $(d)$ holds automotically for $\lambda=\kappa^+$; given any collection $\mathcal{F}$ of $\kappa$-many elements of $\ds\prod_{n<\omega}(\kappa_{n+1}\setminus\kappa_n)$ there exists $f$ such that for each $g \in \mathcal{F}, f(n)> g(n)$ for all large $n$. Thus we can define, by induction on $\alpha< \kappa^+$, an increasing (mod finite) sequence $\langle
f_{\alpha}: \alpha < \kappa^+ \rangle$ in $\ds\prod_{n<\omega}(\kappa_{n+1}\setminus\kappa_n)$.}

and

$(e)$ there exists a club $C \subseteq \lambda$
which avoids points of countable $V$-cofinality.

Then adding $\kappa$-many Cohen reals over $V_1$ produces $\lambda$-many Cohen reals over $V.$
\end{theorem}
\begin{proof}
We consider two cases.

{\bf Case} {\bf $\lambda=\kappa^{+}.$}
Force to add
$\k$-many Cohen reals over $V_1$. Split them into two sequences of
length $\k$ denoted by $\lan r_{\imath} : \imath<\k\ran$ and $\lan r'_{\imath} :
\imath<\k\ran$. Also let $\lan f_\a : \a<\k^+\ran\in V_1$ be an
increasing (mod finite) sequence in $\ds\prod_{n<\om}(\k_{n+1}
\setminus \k_{n})$. Let $\a<\k^+$. We define a real $s_\a$ as
follows:

~~~~~~~~~~~~~~~~~~~~~~~~~~~~~~~~~~~~~~~~~~~~~

{\bf Case 1.} $\a\in C$. Then

\begin{center}
$\forall n<\om$, $s_\a(n)=r_{f_\a(n)}(0)$.
\end{center}

{\bf Case 2.} $\a\not\in C$. Let $\a^*$ and $\a^{**}$ be two
successive points of $C$ so that $\a^*<\a<\a^{**}$. Let $\lan\a_{\imath} :
\imath<\k\ran$ be some fixed enumeration of the interval
$(\a^*,\a^{**})$. Then for some $\imath<\k$,  $\a=\a_{\imath}$.
Let $k(\imath)=\min\{k<\om : r'_{\imath}(k)=1\}$. Set

\begin{center}
 $\forall n<\om$, $s_\a(n)=r_{f_\a(k(\imath)+n)}(0)$.
\end{center}

The following lemma completes the proof.

\begin{lemma} $\lan s_\a:\a<\k^+\ran$ is a sequence of $\k^+$-many Cohen
reals over $V$.
\end{lemma}
\begin{notation} For each set $I$, let ${\MCB}(I)$ be the Cohen forcing
notion for adding $I$-many Cohen reals. Thus ${\MCB}(I)=\{p:p$ is
a finite partial function from $I\times \om$ into 2 $\}$, ordered
by reverse inclusion.
\end{notation}
\begin{proof} First note that $\lan\lan r_{\imath} : \imath<\k\ran, \lan r'_{\imath} :
\imath<\k\ran\ran$ is ${\MCB}(\k)\times {\MCB}(\k)$-generic over $V_1$.
By c.c.c of ${\MCB}(\k^+)$ it suffices to show that for any
countable set $I\sse\k^+$, $I\in V$, the sequence $\lan s_\a:\a\in
I\ran$ is ${\MCB}(I)$-generic over $V$. Thus it suffices to prove
the following:

~~~~~~~~~~~~~~~~~~~~~~~~~~~~~~~~~~~~~~~~~~~~~~~~~~~~~~~~~~~~~~~

$\hspace{1.5cm}$ for every $(p,q)\in {\MCB}(\k)\times {\MCB}(\k)$
and every open dense subset $D\in$

$(*)$ $\hspace{1cm}$ $V$ of ${\MCB}(I)$,~ there is
$(\ov{p},\ov{q})\leq(p,q)$ such~ that $(\ov{p},\ov{q}) \vdash
``  \lan \lusim{s}_\a : \a\in I\ran$

$\hspace{1.6cm}$  extends some element of $D$''.

~~~~~~~~~~~~~~~~~~~~~~~~~~~~~~~~~~~~~~~~~~~~~~~~~~~~~~~~~~~~~~

Let $(p,q)$ and $D$ be as above. For simplicity suppose that
$p=q=\emptyset$. By $(e)$ there are only finitely many $\a^*\in
C$ such that $I\cap[\a^*,\a^{**}) \neq\emptyset$, where
$\a^{**}=\min(C\setminus(\a^*+1))$. For simplicity suppose that
there are two $\a^*_1<\a^*_2$ in $C$ with this property. Let
$n^*<\om$ be such that for all $n\geq n^*$, $f_{\a^*_1}(n)<
f_{\a^*_2}(n)$. Let $p\in {\MCB}(\k)$ be such that
\begin{center}
$\dom(p)=\{\lan \b,0\ran : \exists n<n^*(\b= f_{\a^*_1}(n)$ or $\b=
f_{\a^*_2}(n))\}$.
\end{center}
Then for $n<n^*$ and $j\in\{1,2\}$,
\begin{center}
$(p,\emptyset)\vdash  \lusim{s}_{\a^*_j}(n)=\lusim{r}_{
f_{\a^*_j}(n)} (0)=p(f_{\a^*_j(n)},  0)$
\end{center}
Thus $(p,\emptyset)$ decides $s_{\a^*_1} \upharpoonright n^*$ and $s_{\a^*_2}
\upharpoonright  n^*$. Let $b\in D$ be such that
$  \lan b(\a^*_1),b(\a^*_2)\ran$
extends  $\lan s_{\a^*_1} \upharpoonright n^*,  s_{\a^*_2} \upharpoonright
n^*\ran$,
where $b(\a)$ is defined by $b(\a):\{n:(\a,n)\in \dom(b)\}\lora 2$
and $b(\a)(n)=b(\a,n)$. Let
\begin{center}
$p'=p\cup \ds\bigcup_{j\in\{1,2\}} \{\lan f_{\a^*_j}(n), 0,
b(\a^*_j,n)\ran : n\geq n^*, ( \a^*_j,n ) \in \dom(b)\}$.
\end{center}
Then $p'\in {\MCB}(\k)$ \footnote{This is because for $n\geq n^*, f_{\a^*_1}(n)\neq f_{\a^*_2}(n) $ and for $j\in \{1,2 \}, f_{\a^*_j}(n)\notin \{f_{\a^*_j}(m):m<n   \}, $ thus there are no collisions.} and
\begin{center}
$(p',\emptyset)\vdash `` \langle
\lusim{s}_{\a^*_1},\lusim{s}_{\a^*_2}\ran$ \ \ {extends} \ \ \
$\lan b(\a^*_1),b(\a^*_2)\ran$''.
\end{center}
For $j\in\{1,2\}$, let $\{\a_{j_0},...,\a_{jk_j-1}\}$ be an
increasing enumeration of components of $b$ in the interval
$(\a^*_j,\a^{**}_j)$ (i.e. those $\a\in(\a^*_j,\a^{**}_j)$ such
that $( \a,n ) \in \dom(b)$ for some $n$). For $j \in \{1,2\}$ and
$l < k_j$ let $\a_{jl}=\a_{\imath_{jl}}$ where $\imath_{jl}<\k$ is
the index of $\a_{jl}$ in the enumeration of the interval
$(\a^*_j,\a^{**}_j)$ considered in Case 2 above.  Let $m^*<\om$ be
such that for all $n\geq m^*$, $j\in\{1,2\}$ and $l_j<l'_j<k_j$ we
have
\begin{center}
$f_{\a^*_1}(n) < f_{\a_{1\ell_1}}(n) < f_{\a_{1\ell'_1}}(n) <
f_{\a^*_2}(n) < f_{\a_{2\ell_2}}(n) <  f_{\a_{2\ell'_2}}(n).$
\end{center}
Let
\begin{center}
$\ov{q}=\{\lan \imath_{jl},n,0\ran : j\in\{1,2\}, l<k_j, n<m^*\} \cup \{\lan \imath_{jl},m^*,1\ran : j\in\{1,2\}, l<k_j\}$.
\end{center}
Then $\ov{q}\in {\MCB}(\k)$ and for $j\in \{1,2\}$ and $n<m^*$,
$(\emptyset,\ov{q})\vdash$``$  r'_{\imath_{jl}}(n)=0$ and $  r'_{\imath_{jl}}(m^*)=1
$'', thus $(\emptyset,\ov{q})\vdash  k(j,l)=\min\{k<\om: r'_{\imath_{jl}}(k)=1\}=
m^*$. Let
\begin{center}
$\ov{p}=p'\cup \ds\bigcup_{j\in\{1,2\}} \{\lan f_{\a_{jl}}
(m^{*}+n), 0,b (\a_{jl},n)\ran : l<k_j, (\a_{jl},n) \in \dom(b)\}.$
\end{center}
It is easily seen that $\ov{p}\in {\MCB}(\k)$ is well-defined and
for $j\in\{1,2\}$ and $l<k_j$,
\begin{center}
 $(\ov{p},\ov{q}) \vdash ``
\lusim{s}_{\a_{jl}}$ extends $b(\a_{jl})$''.
\end{center}
Thus
\begin{center}
$(\ov{p},\ov{q})\vdash ``  \langle \lusim{s}_\a: \a\in
I\ran$ extends $b$''.
\end{center}
$(*)$ follows and we are done.
\end{proof}

{\bf Case} {\bf $\lambda > \kappa^{+}.$} Force to add
$\k$-many Cohen reals over $V_1$.
We now construct $\l$-many Cohen reals over $V$ as in the above case
using $C$ and $\lan f_\a : \a<\l\ran$. Case 2 of the definition
of $\lan s_\a : \a <\l\ran$ is now problematic since the
cardinality of an interval $(\a^*,\a^{**})$ (using the above notation) may now be above $\k$ and we have only $\k$-many Cohen reals
to play with. Let us proceed as follows in order to overcome this.

Let us rearrange the Cohen reals as $\lan r_{n,\a}:n<\om,
\a<\k\ran$ and $\lan r_\eta:\eta\in[\k]^{<\om}\ran$. We define by
induction on levels a tree $T\sse[\l]^{<\om}$, its projection
$\pi(T)\sse[\k]^{<\om}$ and for each $n<\om$ and $\a\in L ev_n(T)$
a real $s_\a$. The union of the levels of $T$ will be $\l$ so
$\lan s_\a:\a<\l\ran$ will be defined.

For $n=0$, let $Lev_0(T)=\lan\ran=Lev_0(\pi(T))$.

For $n=1$, let $Lev_1(T)=C, Lev_1(\pi(T))=\{0\}$, i.e.
$\pi(\lan\a\ran)=\lan 0\ran$ for every $\a\in C$. For $\a\in
C$ we define a real $s_\a$ by
\begin{center}
$\forall m<\om$, $s_\a(m)=r_{1,f_\a(m)}(0)$.
\end{center}

Suppose now that $n>1$ and $T\upharpoonright n$ and $\pi(T)\upharpoonright n$ are
defined. We define $Lev_n(T)$, $L ev_n(\pi(T))$ and reals $s_\a$
for $\a\in Lev_n(T)$. Let $\eta\in T\upharpoonright (n-1)$, $\a^*,\a^{**}\in
Suc_T(\eta)$ and $\a^{**}=\min(Suc_T(\eta)\setminus(\a^*+1))$. We then
define $Suc_T(\eta^{\frown}\lan\a^{**}\ran)$ if it is not yet
defined \footnote{Then $Lev_{n}(T)$ will be the union of such $Suc_T(\eta^{\frown}\lan\a^{**}\ran)$'s.}.

~~~~~~~~~~~~~~~~~~~~~~~~~~~~~~~~~~~~~~~~~~~~~~~

{\bf Case A.} $|\a^{**}\setminus \a^*|\leq\k$.

Fix some enumeration $\lan\a_{\imath}:\imath<\rho \leq\k\ran$ of
$\a^{**}\setminus\a^*$. Let

\begin{itemize}
\item $Suc_T(\eta^{\frown}\lan\a^{**}\ran)= \a^{**} \setminus\a^*$,
\item $Suc_T(\eta^{\frown}\lan\a^{**}\ran^{\frown}\lan\a\ran)=\lan\ran$
for $\a\in\a^{**}\setminus\a^*$, \item
$Suc_{\pi(T)}(\pi(\eta^{\frown}\lan\a^{**}\ran))=\rho=|\a^{**}
\setminus\a^*|$, \item
$Suc_{\pi(T)}(\pi(\eta^{\frown}\lan\a^{**}\ran)^{\frown}\lan
\imath\ran)=\lan\ran$ for $\imath<\rho$.
\end{itemize}

Now we define $s_\a$ for $\a\in\a^{**}\setminus\a^*$. Let
$\imath$ be such that $\a=\a_{\imath}$. let
$k=\min\{m<\om:r_{\pi(\eta^{\frown}\lan\a^{**}\ran)^{\frown}\lan\imath\ran}(m)=1\}$,
Finally let

\begin{center}
 $\forall m<\om, s_\a(m)=r_{n,f_\a(k+m)}(0)$.
\end{center}

{\bf Case B.} $|\a^{**} \setminus\a^*|>\k$ and $cf(\a^{**})<\k.$

Let $\rho=cf\a^{**}$ and let $\lan \a^{**}_\nu:\nu<\rho\ran$ be a
normal sequence cofinal in $\a^{**}$ with $\a^{**}_0>\a^*$. Let

\begin{itemize}
\item $Suc_T(\eta^{\frown}\lan\a^{**}\ran)=\{\a^{**}_\nu:\nu<\rho\}$,
\item $Suc_{\pi(T)}(\pi (\eta^{\frown}\lan\a^{**}\ran))=\rho$.
\end{itemize}

 Now we define $s_{\a^{**}_\nu}$ for $\nu<\rho$. Let $k=\min\{m<\om: r_{\pi (\eta^{\frown}\lan\a^{**}\ran)^{\frown}\lan\nu\ran} (m)=1 \}$
 and let
\begin{center}
 $\forall m<\om,s_{\a^{**}_\nu}(m)=r_{n,f_{\a^{**}_\nu}(k+m)}(0)$.
\end{center}

{\bf  Case C.} $cf(\a^{**})>\k$.

Let $\rho$ and $\lan \a^{**}_\nu:\nu<\rho \ran$ be as in Case B.
Let

\begin{itemize}
\item $Suc_T(\eta^{\frown}\lan\a^{**}\ran)= \{\a^{**}_\nu:\nu<\rho\}$, \item $Suc_{\pi(T)}(\pi
(\eta^{\frown}\lan\a^{**}\ran))=\lan 0\ran $.
\end{itemize}

 We define
$s_{\a^{**}_\nu}$ for $\nu<\rho$. Let
$k=\min\{m<\om:r_{\pi(\eta^{\frown}\lan\a^{**}\ran)^{\frown}\lan 0\ran}
(m)=1\}$ and let
\begin{center}
 $\forall m<\om,s_{\a^{**}_\nu} (m)=r_{n,f_{\a^{**}_\nu}(k+m)}(0)$.
\end{center}
By the definition, $T$ is a well-founded tree and
$\ds\bigcup_{n<\om} Lev_n(T)=\l$. The following lemma completes
our proof.

\begin{lemma} $\lan s_\a : \a<\l\ran$ is a sequence of $\l$-many Cohen
reals over $V$.
\end{lemma}
\begin{proof} First note that $\lan\lan r_{n,\a} : n<\om,
\a<\k\ran,\lan r_\eta:\eta\in[\k]^{<\om}\ran\ran$ is
${\MCB}(\om\times\k)\times {\MCB}([\k]^{<\om})$-generic over
$V_1$. By c.c.c of ${\MCB}(\l)$ it suffices to show that for any
countable set $I\sse\l$, $I\in V$, the sequence $\lan s_\a:\a\in
I\ran$ is ${\MCB}(I)$-generic over $V$. Thus it suffices to prove
the following:

~~~~~~~~~~~~~~~~~~~~~~~~~~~~~~~~~~~~~~~~~~~~~~~~~~~~~~~~~~~

$\hspace{1.4cm}$ For every $(p,q)\in {\MCB}(\om\times \k)\times
{\MCB}([\k]^{<\om})$ and every open dense subset

$(*)$$\hspace{1cm}$ $D\in V$ of ${\MCB}(I)$, there is
$(\ov{p},\ov{q})\leq(p,q)$ such that $(\ov{p},\ov{q})\vdash
``  \lan \lusim{s}_\a : \a\in I\ran$

$\hspace{1.5cm}$ extends some element of $D$''.

~~~~~~~~~~~~~~~~~~~~~~~~~~~~~~~~~~~~~~~~~~~~~~~~~~~~~~~~~

Let $(p,q)$ and $D$ be as above. for simplicity suppose that
$p=q=\emptyset$. For each $n<\om$ let $I_n=I\cap Lev_n(T)$. Then
$I_0=\emptyset$ and $I_1=I\cap C$ is finite. For simplicity let
$I_1=\{\a^*_1,\a^*_2\}$ where $\a^*_1<\a^*_2$. Pick $n^*<\om$ such
that for all $n\geq n^*$, $f_{\a^*_1}(n)<f_{\a^*_2}(n)$. Let
 $p_0\in {\MCB}(\om\times\k)$ be such that
\begin{center}
$\dom(p_0)=\{\lan 1,\b,0\ran:\exists n<n^*(\b=f_{\a^*_1}(n)$ or
$\b=f_{\a^*_2}(n))\}$.
\end{center}
 Then for $n<n^*$ and $j\in\{1,2\}$
\begin{center}
$(p_0,\emptyset)\vdash
\lusim{s}_{\a^*_j}(n)=\lusim{r}_{1,f_{\a^*_j}(n)}(0)=p_0(1,f_{\a^*_j}(n),0)
$.
\end{center}
thus $(p_0,\emptyset)$ decides $s_{\a^*_1}\upharpoonright n^*$ and
$s_{\a^*_2}\upharpoonright n^*$. Let $b\in D$ be such that
$  \lan b(\a^*_1), b(\a^*_2)
\rangle$ extends $\lan s_{\a^*_1}\upharpoonright n^*, {s}_{\a^*_2}\upharpoonright
n^*\ran$.
Let
\begin{center}
$p_1=p_0\cup \ds\bigcup_{j\in\{1,2\}} \{ \langle
1,f_{\a^*_j}(n),0,b(\a^*_j,n)\ran:n\geq n^*,(\a^*_j,n)\in \dom(b)\}$.
\end{center}
 Then $p_1\in {\MCB}(\om\times\k)$ is well-defined and letting $q_1=\emptyset$, we have
\begin{center}
$(p_1,q_1)\vdash ``  \lan
\lusim{s}_{\a^*_1},\lusim{s}_{\a^*_2}\ran$ extends $\lan
b(\a^*_1), b(\a^*_2)\ran$''.
\end{center}

For each $n<\om$ let $J_n$ be the set of all components of $b$
which are in $I_n$, i.e. $J_n=\{\a\in I_n:\exists n, (\a,n)\in \dom(b)\}$. We note that $J_0=\emptyset$ and
$J_1=I_1=\{\a^*_1,\a^*_2\}$. Also note that for all but finitely
many $n<\om, J_n=\emptyset$. Thus let us suppose $t<\om$ is such
that for all $n>t$, $J_n=\emptyset$. Let us consider $J_2$. For
each $\a\in J_2$ there are three cases to be considered \footnote{Note that all the action in Cases 1-3 below is happening in the generic extension; in particular we did not yet determine the value of $k_{\alpha}$.}:

~~~~~~~~~~~~~~~~~~~~~~~~~~~~~~~

{\bf Case 1.} There are $\a^*<\a^{**}$ in $Lev_1(T)=C$,
$\a^{**}=\min(C\setminus(\a^*+1))$ such that $|\a^{**}
\setminus\a^*|\leq\k$ and
 $\a\in Suc_T(\lan\a^{**}\ran)=\a^{**}\setminus\a^*$.  Let $\imath_\a$ be the index of $\a$ in the
 enumeration of $\a^{**}\setminus\a^*$ considered in Case $A$ above, and
 let $k_\a=\min\{m<\om:r_{\pi(\lan\a^{**}\ran)^{\frown}\lan \imath_\a\ran}(m)=1\}$. Then
\begin{center}
$\forall m<\om$, $s_\a(m)=r_{2,f_\a(k_\a+m)}(0)$.
\end{center}

{\bf Case 2.} There are $\a^*<\a^{**}$  as above such that
$|\a^{**}\setminus\a^*|>\k$ and $\rho=cf\a^{**}<\k$. Let
$\lan\a^{**}_\nu:\nu<\rho\ran$ be as in Case B. Then
$\a=\a^{**}_{\nu_\a}$ for some $\nu_\a<\rho$ and if
$k_\a=\min\{m<\om:r_{\pi(\lan\a^{**}\ran)^{\frown}\lan
\nu_\a\ran}(m)=1\}$. Then
\begin{center}
$\forall m<\om$, $s_\a(m)=r_{2,f_\a(k_\a+m)}(0)$.
\end{center}

{\bf Case 3.} There are $\a^*<\a^{**}$  as above such that
$\rho=cf\a^{**}>\k$. Let  $\lan\a^{**}_\nu:\nu<\rho\ran$ be as in
Case C. Then $\a=\a^{**}_{\nu_\a}$ for some $\nu_\a<\rho$ and
if  $k_\a=\min\{m<\om:r_{\pi(\lan\a^{**}\ran)^{\frown}\lan
0\ran}(m)=1\}$, then
\begin{center}
$\forall m<\om$, $s_\a(m)=r_{2,f_\a(k_\a+m)}(0)$.
\end{center}

Let $m^*<\om$ be such that for all $n\geq m^*$ and $\a<\a'$ in
$J_1\cup J_2$, $f_\a(n)<f_{\a'}(n)$. Let
\begin{center}
 $q_2=\{\lan \eta,n,0\ran :
n<m^*, \exists \a\in J_2(\eta=\pi(\lan\a^{**}\ran)^{\frown}\lan
i_\a\ran$ or

$\hspace{4.95cm}$  $\eta=\pi(\lan\a^{**}\ran)^{\frown}
\lan\nu_\a\ran$ or

$\hspace{5.5cm}$  $\eta=\pi(\lan\a^{**}\ran)^{\frown}\lan 0\ran )
\}$

$\hspace{1.6cm}$$\cup$ $\{\lan \eta, m^*,1\ran :
 \exists \a\in J_2(\eta=\pi(\lan\a^{**}\ran)^{\frown}\lan
i_\a\ran$ or

$\hspace{4.95cm}$  $\eta=\pi(\lan\a^{**}\ran)^{\frown}
\lan\nu_\a\ran$ or

$\hspace{5.5cm}$  $\eta=\pi(\lan\a^{**}\ran)^{\frown}\lan 0\ran )
\}$

\end{center}
 Then
$q_2\in {\MCB}([\k]^{<\om})$ is well-defined and for each $\a\in
J_2$, $(\phi,q_2)\vdash   k_\a = m^* $. Let
\begin{center}
$p_2=p_1\cup\{\lan 2, f_\a(m^{*}+m),0,b(\a,m)\ran:\a\in J_2,
(\a,m)\in \dom(b)\}$.
\end{center}
 Then $p_2\in {\MCB}(\om\times \k)$ is
well-defined, $(p_2,q_2)\leq(p_1,q_1)$ and for $\a\in J_2$ and
$m<\om$ with $(\a,m)\in \dom(b)$,
\begin{center}
$(p_2,q_2)\vdash
\lusim{s}_\a(m)=\lusim{r}_{2,f_\a(k_\a+m)}(0)=p_2(2,f_\a(k_\a+m),0)=b(\a,m)=b(\a)(m)
$,
\end{center}
thus $(p_2,q_2)\vdash `` \lusim{s}_\a$ extend $b(\a)
$'' and hence
\begin{center}
$(p_2,q_2)\vdash ``  \lan\lusim{s}_\a: \a\in J_1\cup
J_2\ran $ extends $\lan b(\a):\a\in J_1\cup J_2\ran $''.
\end{center}

By induction suppose that we have defined $(p_1,q_1)\geq
(p_2,q_2)\geq ... \geq (p_j,q_j)$ for $j<t$, where for $1\leq
i\leq j$,
\begin{center}
$(p_i,q_i)\vdash `` \lan\lusim{s}_\a: \a\in J_1\cup
...\cup J_i\ran $ extends $\lan b(\a):\a\in J_1\cup ...\cup
J_i\ran $''.
\end{center}
We define $(p_{j+1},q_{j+1})\leq (p_{j},q_{j})$ such that for each
$\a\in J_{j+1}, (p_{j+1},q_{j+1}) \vdash ``  \lusim{s}_\a$
extends $b(\a)$''.

 Let $\a\in J_{j+1}$. Then we can find
$\eta\in T\upharpoonright j$ and $\a^*<\a^{**}$ such that $\a^*,\a^{**}\in
Suc_T(\eta)$, $\a^{**}=\min(Suc_T(\eta)\setminus(\a^*+1))$ and
$\a\in Suc_T(\eta^{\frown}\lan\a^{**}\ran)$. As before there are
three cases to be considered \footnote{Again note that all the action in Cases 1-3 below is happening in the generic extension.}:

~~~~~~~~~~~~~~~~~~~~~~~~~~~~~~~~~~~

{\bf Case 1.} $|\a^{**}\setminus\a^*|\leq\k$. Then let $i_\a$ be
the index of $\a$ in the enumeration of $\a^{**}\setminus\a^*$
considered in Case A and let $k_\a=\min\{m<\om:
r_{\pi(\eta^{\frown}\lan \a^{**}\ran)^{\frown} \lan i_\a\ran}(m)=1\}$.
Then
\begin{center}
$\forall m<\om, s_\a(m)=r_{j+1,f_\a(k_\a+m)}(0)$.
\end{center}

{\bf Case 2.} $|\a^{**}\setminus\a^*|>\k$ and
$\rho=cf\a^{**}<\k$. Let $\lan\a^{**}_\nu:\nu<\rho\ran$ be as in
 Case B and let ${\nu_\a}<\rho$ be such that $\a=\a^{**}_{\nu_\a}$.
 Let  $k_\a=\min\{m<\om: r_{\pi(\eta^{\frown}\lan \a^{**}\ran)^{\frown} \lan \nu_\a\ran}(m)=1\}$. Then
\begin{center}
 $\forall m<\om, s_\a(m)=r_{j+1,f_\a(k_\a+m)}(0)$.
\end{center}

{\bf Case 3.} $\rho=cf\a^{**}>\k$. Let
$\lan\a^{**}_\nu:\nu<\rho\ran$ be as in Case C. Let ${\nu_\a}<\rho$ be such that $\a=\a^{**}_{\nu_\a}$ and let
$k_\a=\min\{m<\om: r_{\pi(\eta^{\frown}\lan \a^{**}\ran)^{\frown} \lan
0\ran}(m)=1\}$. Then
\begin{center}
$\forall m<\om, s_\a(m)=r_{j+1,f_\a(k_\a+m)}(0)$.
\end{center}
Let $m^*<\om$ be such that for all $n\geq m^*$ and $\a<\a'$ in
$J_1\cup...\cup J_{j+1}, f_\a(n)< f_{\a'}(n)$. Let

$\hspace{1cm}$$q_{j+1}=q_j\cup\{\lan \ov{\eta},n,0\ran: n<m^*,
\exists \a\in J_{j+1}$ (for some unique $\eta\in T\upharpoonright j,$

$\hspace{4.4cm}$ $\a^{**}\in Suc_T(\eta)$, we have $\a\in
Suc_T(\eta^{\frown}\lan\a^{**}\ran)$

$\hspace{7.1cm}$ and $(\overline{\eta}=\pi(\eta^{\frown}\lan
\a^{**}\ran)^{\frown} \lan i_\a\ran$

 $\hspace{7.4cm}$ or $\overline{\eta}=\pi(\eta^{\frown}\lan
\a^{**}\ran)^{\frown} \lan \nu_\a\ran$

$\hspace{7.cm}$ or $\overline{\eta}=(\pi(\eta^{\frown}\lan
\a^{**}\ran)^{\frown} \lan 0\ran))\}$

$\hspace{3.5cm}$ $\cup \{\lan \ov{\eta},m^*,1\ran:
\exists \a\in J_{j+1}$ (for some unique $\eta\in T\upharpoonright j,$

$\hspace{4.4cm}$ $\a^{**}\in Suc_T(\eta)$, we have $\a\in
Suc_T(\eta^{\frown}\lan\a^{**}\ran)$

$\hspace{7.1cm}$ and $(\overline{\eta}=\pi(\eta^{\frown}\lan
\a^{**}\ran)^{\frown} \lan i_\a\ran$

 $\hspace{7.4cm}$ or $\overline{\eta}=\pi(\eta^{\frown}\lan
\a^{**}\ran)^{\frown} \lan \nu_\a\ran$

$\hspace{7.cm}$ or $\overline{\eta}=(\pi(\eta^{\frown}\lan
\a^{**}\ran)^{\frown} \lan 0\ran))\}.$

 It is easily seen that $q_{j+1}\in {\MCB}([\k]^{<\om})$ and
for each $\a\in J_{j+1}$, $(\phi,q_{j+1})\vdash  k_\a =
m^* $. Let

\begin{center}
 $p_{j+1}=p_j\cup\{\lan j+1,
f_\a(m^{*}+m),0,b(\a,m)\ran: \a\in J_{j+1}, (\a,m)\in \dom(b)\}$.
\end{center}

Then $p_{j+1}\in {\MCB}(\om\times \k)$ is well-defined and
$(p_{j+1},q_{j+1})\leq(p_j,q_j)$ and for $\a\in J_{j+1}$ we have

\begin{center}
$(p_{j+1},q_{j+1})\vdash  \lusim{s}_\a(m)
=\lusim{r}_{j+1,f_\a(k_\a+m)}(0)= p_{j+1}({j+1},f_\a(k_\a+m),0)=
b(\a,m)=b(\a)(m) $.
\end{center}

Thus $(p_{j+1},q_{j+1})\vdash ``  \lusim{s}_\a$ extends
$b(\a) $''. Finally let $(\ov{p},\ov{q})=(p_t,q_t)$. Then
for each component $\a$ of $b$,
\begin{center}
$(\ov{p},\ov{q}) \vdash `` \lusim{s}_\a$ extends $b(\a)
$''.
\end{center}
Hence

\begin{center}
 $(\ov{p},\ov{q})
\vdash ``  \lan \lusim{s}_\a: \a\in I\ran$ extends $b
$''.
\end{center}

$(*)$ follows and we are done.
\end{proof}
Theorem 2.1 follows.
\end{proof}

We now give several applications of the above theorem.

\begin{theorem} Suppose that $V$ satisfies $GCH$,
$\k=\ds\bigcup_{n<\om} \k_n$ and $\ds\bigcup_{n<\om} o(\k_n)=\k$
(where $o(\k_n)$ is the Mitchell order of $\k_n$). Then there
exists a cardinal preserving generic extension $V_1$ of $V$
satisfying $GCH$ and having the same reals as $V$ does, so that
adding $\k$-many Cohen reals over $V_1$ produces $\k^+$-many Cohen
reals over $V$.
\end{theorem}
\begin{proof} Rearranging the sequence $\lan \k_n:n<\om\ran$ we may
assume that $o(\k_{n+1})>\k_n$  for each $n<\om$. Let $0<n<\om$.
By [Mag 1], there exists a forcing notion ${\MPB}_n$ such that:

\begin{itemize}
\item Each condition in ${\MPB}_n$ is of the form $(g,G)$, where
$g$ is an increasing function from a finite subset of $\k^+_n$
into $\k_{n+1}$ and $G$ is a function from $\k^+_n \setminus \dom(g)$ into
$\mathcal{P}(\k_{n+1})$ such that for each $\alpha \in \dom(G), G(\alpha)$ belongs to a suitable normal measure \footnote{In fact if $\alpha > \max(\dom(g))$, then $G(\alpha)$ belongs to a normal measure on $\kappa_{n+1}$, and if $\alpha < \max(\dom(g))$, then $G(\alpha)$ belongs to a normal measure on $g(\beta)$ where $\beta=\min(\dom(g)\setminus \alpha)$.}. We may also assume that conditions have
no parts below or at $\k_n$, and sets of measure one are like this
as well. \item Forcing with ${\MPB}_n$ preserves cardinals and the
$GCH$, and adds no new subsets to $\k_n$. \item If $G_n$ is
${\MPB}_n$-generic over $V$, then in $V[G_n]$ there is a normal
function $g^*_n:\k^+_n\lora \k_{n+1}$ such that $ran(g^*_n)$ is a
club subset of $\k_{n+1}$ consisting of measurable cardinals of
$V$ such that $V[G_n]=V[g^*_n]$.
\end{itemize}

Let ${\MPB}^*=\ds\prod_{n<\om}{\MPB}_n$, and let
\begin{center}
${\MPB}=\{\lan\lan g_n,G_n\ran:
n<\om\ran\in{\MPB}^*:g_n=\emptyset$,  for all but finitely many
$n\}$.
\end{center}

Then ${\MPB}$ satisfies the $\kappa^{+}-c.c.$\footnote{This is because any two conditions $\lan\lan g_n,G_n\ran:
n<\om\ran$ and $\lan\lan g_n,H_n\ran:
n<\om\ran$ in ${\MPB}$  are compatible.} and using simple modification of arguments from [Mag 1,2] we can
show that forcing with ${\MPB}$ preserves cardinals and the $GCH$.
 Let $G$ be ${\MPB}$-generic over $V,$ and let $g^*_n:\k^+_n\lora \k_{n+1}$ be the generic
 function added by the part of the forcing corresponding to ${\MPB}_n$, for $0<n<\om$.
 Let $X=\ds\bigcup_{0<n<\om}((ran(g^*_n)\setminus \k^+_n)\cup\{\k_{n+1}\})$ and let $g^*:\k\lora\k$ be an
 enumeration of $X$ in increasing order. Then $X=ran(g^*)$ is club in $\k$ and consists entirely of
 measurable cardinals of $V$. Also $V[G]=V[g^*]$.

Working in $V[G]$, let ${\MQB}$ be the usual forcing notion for
adding a club subset of $\k^+$ which avoids points of countable
$V$-cofinality. Thus ${\MQB}=\{p:p$ is a closed bounded subset of
$\k^+$ and avoids points of countable $V$-cofinality$\}$, ordered
by end extension. Let $H$ be
 ${\MQB}$-generic over $V[G]$ and $C=\bigcup\{p:p\in H\}$.

\begin{lemma} $(a)$ (${\MQB},\leq$) satisfies the $\k^{++}$-c.c.,

$(b)$  (${\MQB},\leq$) is $<\k^+$-distributive,

$(c)$ $C$ is a club subset of $\k^+$ which avoids points of
countable $V$-cofinality.
\end{lemma}
 $(a)$ and $(c)$ of the above lemma are trivial. For use later we prove a more general version of $(b).$

\begin{lemma} Let $V\sse W$, let $\nu$  be regular in $W$ and suppose
that:

$(a)$  $W$ is a $\nu$-c.c. extension of $V$,

$(b)$  For every $\l<\nu$ which is regular in $W$, there is
$\tau<\nu$ so that $cf^W(\tau)=\l$ and $\tau$ has a club subset in
$W$ which avoids points of countable $V$-cofinality.

In $W$ let ${\MQB}=\{p\sse\nu: p$ is closed and bounded in $\nu$
and avoids points of countable $V$-cofinality$\}$. Then in $W$,
${\MQB}$ is $<\nu$-distributive.
\end{lemma}
\begin{proof} This lemma first appeared in [G-N-S]. We prove it for
completeness. Suppose that $W=V[G]$, where $G$ is ${\MPB}$-generic
over $V$ for a $\nu$-c.c. forcing notion ${\MPB}$. Let $\l<\nu$ be
regular,  $q\in {\MQB}, \lusim{f}\in W^{{\MQB}}$ and

\begin{center}
$q\vdash  \lusim{f}:\l\lora ON$.
\end{center}

We find an extension of $q$ which decides $\lusim{f}$. By $(b)$ we can find
$\tau<\nu$ and $g:\l\lora\tau$ such that $cf^W(\tau)=\l$, $g$ is
normal and $C=ran(g)$ is a club of $\tau$ which avoids points of
countable $V$-cofinality.

In $W$, let $\theta>\nu$ be large enough regular. Working in $V$,
let $\ov{H} \prec V_\theta$  and $R:\tau\lora ON$ be such that

\begin{itemize}
\item $Card(\ov{H})<\nu$, \item $\ov{H}$ has $\l,\tau,\nu,{\MPB}$
and ${\MPB}$-names for $p,{\MQB},\lusim{f},g$ and $C$ as elements,
\item $ran(R)$ is cofinal in $\sup(\ov{H}\cap\nu),$ \item $R\upharpoonright
\beta\in\ov{H}$ for each $\beta<\tau$.
\end{itemize}

Let $H=\ov{H}[G]$. Then $\sup(H\cap\nu)=\sup(\ov{H}\cap\nu)$, since
${\MPB}$ is $\nu$-c.c., $H \prec V_{\theta}^{W}$ and if
$\g=sup(H\cap\nu)$, then $cf^W(\g)=cf^W(\tau)=\l$. For $\a<\l$ let
$\g_\a=R(g(\a))$. Then

\begin{itemize}
\item $\lan \g_\a: \a<\l\ran\in W$ is a normal sequence cofinal in
$\g$, \item $\lan \g_\a: \a<\b\ran\in H$ for  each $\b<\l$, since
$R\upharpoonright g(\b)\in\ov{H}$, \item $cf^V(\g_\a)=cf^V(g(\a))\neq \om$
for each $\a<\l$, since $R$ is normal and $g(\a)\in C$.
\end{itemize}

Let $D=\{\g_\a:\a<\l\}$. We define by induction a sequence $\lan
q_\eta:\eta<\l\ran$ of conditions in ${\MQB}$ such that for each
$\eta<\l$

\begin{itemize}
\item $q_0=q$, \item $q_\eta\in H$, \item $q_{\eta+1}\leq
q_{\eta}$, \item $q_{\eta+1}$ decides $\lusim{f}(\eta)$, \item
$D\cap(\max(q_\eta), \max(q_{\eta+1}))\neq\emptyset,$ \item
$q_\eta=\ds\bigcup_{\rho<\eta} q_\rho\cup\{\d_\eta\}$,  where
$\d_\eta=\ds\sup_{\rho<\eta}(\max (q_\rho))$, if $\eta$ is a limit
ordinal.
\end{itemize}

We may further suppose that
\begin{itemize}
 \item $q_{\eta}$'s are chosen in a uniform way (say via a well-ordering which is built in to $\ov{H}$).
\end{itemize}

We can define such a sequence using the facts that $H$ contains
all initial segments of $D$ and that $\d_\eta\in D$ for every
limit ordinal $\eta<\l$ (and hence $cf^V(\d_\eta)\neq\om$).

Finally let $q_{\lambda}=\ds\bigcup_{\eta<\l} q_\eta\cup\{\d_\l\}$, where
$\d_\l=\ds\sup_{\eta<\l}(\max (q_\eta))$. Then $\d_\l\in
D\cup\{\g\}$, hence $cf^V(\d_\l)\neq\om$. It follows that
$q_{\lambda}\in{\MQB}$ is well-defined. Trivially $q_{\lambda}\leq q$ and $q_{\lambda}$ decides
$\lusim{f}$. The lemma follows.
\end{proof}

 Let $V_1=V[G*H]$. The following is
obvious

\begin{lemma} $(a)$ $V$ and $V_1$ have the same cardinals and reals,

$(b)$ $V_1\models GCH$,
\end{lemma}
It follows from Theorem 2.1 that adding $\kappa$-many Cohen reals over $V_1$ adds $\kappa^{+}$-many Cohen reals over $V.$ This concludes the proof of Theorem 2.5.
\end{proof}

Let us show that some large cardinals are needed for the previous
result.

\begin{theorem} Assume that $V_1\supseteq V$ and $V_1$ and $V$ have the same
cardinals and reals. Suppose that for some uncountable cardinal
$\k$ of $V_1$, adding $\k$-many Cohen reals to $V_1$ produces
$\k^+$-many Cohen reals to $V$. Then in $V_1$ there is an inner
model with a measurable cardinal.
\end{theorem}
\begin{proof} Suppose on the contrary that in $V_1$ there is no inner
model with a measurable cardinal. Thus by Dodd-Jensen covering
lemma (see [D-J 1,2]) $(\mathcal{K}(V_1),V_1)$ satisfies the covering lemma, where $\mathcal{K}(V_1)$ is the Dodd-Jensen core model as computed in $V_1$.
\begin{claim}
$\mathcal{K}(V)=\mathcal{K}(V_1)$
\end{claim}
\begin{proof}
The claim is well-known and follows from the fact that $V$ and $V_1$ have the same cardinals. We present a proof for completeness \footnote{Our proof is the same as in the proof of [Sh 2, Theorem VII. 4.2(1)].}. Suppose not. Clearly $\mathcal{K}(V) \subseteq \mathcal{K}(V_1),$ so let $A \subseteq \alpha, A\in \mathcal{K}(V_1), A\notin \mathcal{K}(V).$ Then there is a mouse of $\mathcal{K}(V_1)$ to which $A$ belongs, hence there is such a mouse of $\mathcal{K}(V_1)$-power $\alpha.$ It then follows that for every limit cardinal $\lambda>\alpha$ of $V_1$ there is a mouse with critical point $\lambda$ to which $A$ belongs, and the filter is generated by end segments of
\begin{center}
$\{\chi: \chi <\l, \chi$ a cardinal in $V_1 \}.$
\end{center}

As $V$ and $V_1$ have the same cardinals, this mouse is in $V$, hence in $\mathcal{K}(V).$
\end{proof}

Let us denote this common core model by $\mathcal{K}$.  Then $\mathcal{K}\sse V$, and hence $(V, V_1)$ satisfies the
covering lemma. It follows that $([\k^+]^{\leq\om_1})^V$ is
unbounded in $([\k^+]^{\leq\om})^{V_1}$ and since
$\om_1^V=\om_1^{V_1}$, we can easily show that
$([\k^+]^{\leq\om})^V$ is unbounded in $([\k^+]^{\leq\om})^{V_1}$.
Since $V_1$ and $V$ have the same reals, $([\k^+]^{\leq\om})^V=
([\k^+]^{\leq\om})^{V_1}$ and we get a contradiction.
\end{proof}

If we relax our assumptions, and allow some cardinals to collapse,
then no large cardinal assumptions are needed.

\begin{theorem} $(a)$ Suppose $V$ is a model of $GCH$. Then there is a generic
extension $V_1$ of $V$ satisfying $GCH$ so that
the only
cardinal of $V$ which is collapsed in $V_1$ is $\aleph_1$ and such that
adding
$\aleph_\om$-many Cohen reals to $V_1$ produces
$\aleph_{\om+1}$-many of them over $V$.

$(b)$ Suppose $V$ satisfies $GCH$. Then there is a generic
extension $V_1$ of $V$ satisfying $GCH$ and having the same reals
as $V$ does, so that
the
only cardinals of $V$ which are collapsed in $V_1$ are $\aleph_2$
and $\aleph_3$ and
 such that adding $\aleph_\om$-many Cohen reals to $V_1$
produces $\aleph_{\om+1}$-many of them over $V$.
\end{theorem}
\begin{proof} $(a)$ Working in $V$, let
${\MPB}=Col(\aleph_0,\aleph_1)$ and let $G$ be ${\MPB}$-generic
over $V$. Also let $S=\{\a<\om_2:cf^V(\a)=\om_1\}$. Then $S$
remains stationary in $V[G]$. Working in $V[G]$, let ${\MQB}$ be
the standard forcing notion for adding a club subset of $S$ with
countable conditions, and let $H$ be ${\MQB}$-generic over $V[G]$.
Let $C=\bigcup H$. Then $C$ is a club subset of
$\om_1^{V[G]}=\om_2^V$ such that $C\sse S$, and in particular $C$
avoids points of countable $V$-cofinality. Working in $V[G*H],$
let
\begin{center}
${{\MRB}}=\lan \lan {\MPB}_\nu:\aleph_2 \leq \nu \leq
\aleph_{\om+2},\nu$ regular $\ran,\lan
\lusim{{\MQB}}_\nu:\aleph_2\leq \nu\leq\aleph_{\om+1},\nu$ regular
$\ran\ran$
\end{center}
 be the Easton support iteration by letting
$\lusim{{\MQB}}_\nu$ name the poset $\{p\ss\nu:p$ is closed and
bounded in $\nu$ and avoids points of countable $V$-cofinality$\}$
as defined in $V[G*H]^{{\MPB}_\nu}$. Let
\begin{center}
 $K=\lan\lan G_\nu:
\aleph_2\leq\nu\leq\aleph_{\om+2},\nu$ regular $\ran,\lan H_\nu:
\aleph_2\leq\nu\leq\aleph_{\om+1},\nu$ regular $\ran\ran$
\end{center}
 be ${\MRB}$-generic over $V[G*H]$ (i.e $G_\nu$ is
${\MPB}_\nu$-generic over $V[G*H]$ and $H_\nu$ is
${\MQB}_\nu=\lusim{{\MQB}}_\nu[G_\nu]$-generic over
$V[G*H*G_\nu])$. Then

\begin{lemma} $(a)$ ${\MPB}_\nu$ adds a club disjoint from $\{\a<\l:
cf^V(\a)=\om\}$ for each regular $\l\in(\aleph_1,\nu)$,

$(b)$ (By 2.7) $V[G*H*G_\nu]\models ``  {\MQB}_\nu$ is
$<\nu$-distributive'',

$(c)$ $V[G*H]$ and $V[G*H*K]$ have the same cardinals and reals,
and satisfy $GCH$,

$(d)$ In $V[G*H*K]$ there is a club subset $C$ of $\aleph_{\om+1}$
which avoids points of countable $V$-cofinality.
\end{lemma}

Let $V_1=V[G*H*K]$.  By above results, $V_1$ satisfies $GCH$ and
the only cardinal of $V$ which is collapsed in $V_1$ is
$\aleph_1$. The proof of the fact that adding $\aleph_\om$-many
Cohen reals over $V_1$ produces $\aleph_{\om+1}$-many of them
over $V$ follows from Theorem 2.1.

$(b)$ Working in $V$, let ${\MPB}$ be the following version of
Namba forcing:
\begin{center}
 ${\MPB}=\{T\sse\om_2^{<\om}: T$ is a tree and for
every $s\in T$, the set $\{t\in T: t\supset s\}$  has size
$\aleph_2\}$
\end{center}
 ordered by inclusion. Let $G$ be ${\MPB}$-generic
over $V$. It is well-known that forcing with ${\MPB}$ adds no new
reals, preserves cardinals $\geq\aleph_4$ and that
$|\aleph_2^V|^{V[G]}= |\aleph_3^V|^{V[G]}=\aleph_1^{V[G]}=
\aleph_1^V$ (see [Sh 1]). Let $S=\{\a< \om_3:cf^V(\a)=\om_2\}$.

\begin{lemma} $S$ remains stationary in $V[G]$.
\end{lemma}
\begin{proof} See [Ve-W, Lemma 3].
\end{proof}

Now the rest of the proof is exactly as in $(a)$.

The Theorem follows
\end{proof}

By the same line but using stronger initial assumptions, adding
$\k$-many Cohen reals may produce $\l$-many of them for $\l$ much
larger than $\k^+$.

\begin{theorem} Suppose that $\k$ is a strong cardinal, $\l\geq\k$ is
regular and $GCH$ holds. Then there exists a cardinal preserving
generic extension $V_1$ of $V$ having the same reals as $V$ does,
so that adding $\k$-many Cohen reals over $V_1$ produces $\l$-many
of them over $V$.
\end{theorem}
\begin{proof} Working in $V$,
build for each $\delta$ a measure sequence $\vec{u}_{\delta}$ from a $j$ witnessing ``$\k$ is $\delta$-strong'' out to the first weak repeat point. Find $\vec{u}$ such that $\vec{u}=\vec{u}_{\delta}$ for unboundedly many $\delta.$
 Let
${\MRB}_{\vec{u}}$ be the corresponding Radin forcing notion and
let $G$ be  ${\MRB}_{\vec{u}}$-generic over $V$. Then

\begin{lemma} $(a)$ Forcing with ${\MRB}_{\vec{u}}$ preserves cardinals
and the $GCH$ and adds no new reals,

$(b)$ In $V[G]$, there is a club $C_\k\sse \k$ consisting of
inaccessible cardinals of $V$ and $V[G]=V[C_\k]$,

$(c)$ $\k$ remains strong in $V[G]$.
\end{lemma}
\begin{proof} See [Git 2] and [Cu].
\end{proof}

Working in $V[G]$, let
\begin{center}
$E=\lan\lan U_\a : \a<\l\ran , \langle \pi_{\a\b}: \a
\leq_E\b\ran\ran$
\end{center}
be a nice system satisfying conditions (0)-(9) in [Git 2, page
37]. Also let
\begin{center}
 ${\MRB}=\lan\lan
{\MPB}_\nu : \k^+\leq\nu\leq\l^+,\nu$ regular
$\ran,\lan\lusim{{\MQB}}_\nu:\k^+\leq\nu\leq\l,\nu$ regular
$\ran\ran$
\end{center}
be the Easton support iteration by letting $\lusim{{\MQB}}_\nu$
name the poset $\{p\sse\nu:p$ is closed and bounded in $\nu$ and
avoids points of countable $V$-cofinality$\}$ as defined in
$V[G]^{{\MPB}_\nu}$. Let
\begin{center}
$K=\lan\lan G_\nu :\k^+\leq\nu\leq\l^+,\nu $ regular $\rangle,
\langle H_\nu:\k^+\leq\nu\leq\l, \nu$ regular $\ran\ran$
\end{center}
 be
${\MRB}$-generic over $V[G]$. Then

\begin{lemma} $(a)$ ${\MPB}_\nu$ adds a club disjoint form $\{\a<\d:
cf^V(\a)=\om\}$ for each regular $\d\in(\k,\nu)$,

$(b)$ (By 2.7) $V[G*G_\nu]\models ``  {\MQB}_\nu=
\lusim{{\MQB}}_\nu [G_\nu]$  is  $<\nu$-distributive'',

$(c)$ $V[G]$ and $V[G*K]$ have the same cardinals, and satisfy
$GCH$,

$(d)$ ${\MRB}$ is $\leq\k$-distributive, hence forcing with
${\MRB}$ adds no new $\k$-sequences,

$(e)$ In $V[G*K]$, for each regular cardinal $\k\leq\nu\leq\l$
there is a club $C_\nu\sse\nu$ such that $C_\nu$ avoids points of
countable $V$-cofinality.
\end{lemma}

By 2.16.$(d)$, $E$ remains a nice system in $V[G*K]$, except that
the condition (0) is replaced by $(\l,\leq_E)$ is $\k^+$-directed
closed. Hence working in $V[G*K]$, by
 results of [Git-Mag 1,2] and [Mer], we can find a forcing notion $S$ such that if $L$ is $S$-generic over $V[G*K]$ then

\begin{itemize}
\item $V[G*K]$ and $V[G*K*L]$ have the same cardinals and reals,
\item In $V[G*K*L]$, $2^\k=\l$, $cf(\k)=\aleph_0$ and there is an
increasing sequence $\lan \k_n: n<\om\ran$ of regular cardinals
cofinal in $\k$ and an increasing (mod finite) sequence $\lan
f_\a:\a<\l\ran$ in $\ds\prod_{n<\om}(\k_{n+1}\setminus\k_{n})$.
\end{itemize}

Let $V_1=V[G*K*L]$. Then $V_1$ and $V$ have the same cardinals and
reals.  The fact that adding $\k$-many Cohen reals over $V_1$
produces $\l$-many Cohen reals over $V$ follows from Theorem 2.1.
\end{proof}

If we allow many cardinals between $V$ and $V_1$ to collapse, then
using [Git-Mag 1,Sec 2] one can obtain the following

\begin{theorem} Suppose that there is a strong cardinal and $GCH$ holds. Let
$\a<\om_1$. Then there is a model $V_1\supset V$ having the same
reals as $V$ and satisfying $GCH$ below $\aleph^{V_1}_{\om}$ such
that adding $\aleph^{V_1}_{\om}$-many Cohen reals to $V_1$
produces $\aleph^{V_1}_{\a+1}$-many of them over $V$.
\end{theorem}
\begin{proof}
Proceed as in  Theorem 2.14 to produce the model $V[G*K].$ Then  working in $V[G*K]$,
 we can find a forcing notion $S$ such that if $L$ is $S$-generic over $V[G*K]$ then
 \begin{itemize}
\item $V[G*K]$ and $V[G*K*L]$ have the same reals,
\item In $V[G*K*L]$, cardinals $\geq \k$ are preserved, $\k=\aleph_{\omega}, GCH$ holds below $\aleph_{\omega},  2^\k=\aleph_{\alpha+1}$  and there is an
 increasing (mod finite) sequence $\lan
f_\b:\b<\aleph_{\a+1}\ran$ in $\ds\prod_{n<\om}(\aleph_{n+1}\setminus\aleph_{n})$.
\end{itemize}
 Let $V_1=V[G*K*L]$. Then $V_1$ and $V$ have the same
reals.  The fact that adding $\aleph_{\omega}^{V_1}$-many Cohen reals over $V_1$
produces $\aleph_{\alpha+1}^{V_1}$-many Cohen reals over $V$ follows from Theorem 2.1.

\end{proof}

\section{Models with the same cofinality function but different reals}

This section is completely devoted to the proof of the following
theorem.

\begin{theorem} Suppose that $V$ satisfies $GCH$. Then there is a cofinality
preserving generic extension $V_1$ of $V$ satisfying $GCH$ so that
adding a Cohen real over $V_1$ produces $\aleph_1$-many Cohen
reals over $V$.
\end{theorem}
The basic idea of the proof will be to split $\om_1$
into $\om$ sets such that none of them will contain an infinite
set of $V$. Then something like in section 2 will be used for
producing Cohen reals. It turned out however that just not
containing an infinite set of $V$ is not enough. We will use a
stronger property. As a result the forcing turns out to be more
complicated. We are now going to define the forcing sufficient for
proving the theorem. Fix a nonprincipal ultrafilter $U$ over
$\om$.

\begin{definition} Let $(\MPB_U,\leq,\leq^*)$ be the Prikry (or in this
context Mathias) forcing with $U$, i.e.
\begin{itemize}
\item $\MPB_U=\{\lan s,A\ran\in [\om]^{<\om}\times U: \max(s)<\min(A)\},$ \item $\lan t,B\ran\leq \lan s,A\ran \Longleftrightarrow t$
end extends $s$ and $(t\setminus s)\cup B\sse A$, \item $\lan
t,B\ran\leq^* \lan s,A\ran \Longleftrightarrow t=s$ and $B\sse A$.
\end{itemize}
\end{definition}

We call $\leq^*$ a direct or $*$-extension. The following are the
basic facts on this forcing that will be used further.

\begin{lemma} $(a)$ The generic object of $\MPB_U$ is generated by a real,

$(b)$ $(\MPB_U,\leq)$ satisfies the $c.c.c.$,

$(c)$ If $\lan s,A\ran \in \MPB_U$ and $b\sse\om\setminus (\max(s)+1)$ is finite, then there is a $*$-extension of $\lan s,A\ran$,
forcing the generic real to be disjoint to $b$.
\end{lemma}
\begin{proof} $(a)$ If $G$ is $\MPB_U$-generic over $V$, then let
$r=\bigcup\{s:\exists A, \lan s,A\ran\in G\}$. $r$ is a real and
$G=\{\lan s,A\ran\in \MPB_U:r$ end extends $s$ and $r\setminus
s\sse A\}$.

$(b)$ Trivial using the fact that for $\lan s,A\ran,\lan
t,B\ran\in \MPB_U$, if $s=t$ then $\lan s,A\ran$ and $\lan
t,B\ran$ are compatible.

$(c)$ Consider $\lan s,A\setminus (\max(b)+1)\ran$.
\end{proof}

We now define our main forcing notion.

\begin{definition} $p\in\MPB$ iff $p=\lan p_0,\lusim{p}_1\ran$ where\\
(1) $p_0\in\MPB_U$,\\
(2) $\lusim{p}_1$ is a $\MPB_U$-name such that for some
$\a<\om_1$, $p_0\vdash   \lusim{p}_1 : \a\lora\om
$ and such that
the following hold

(2a) For every $\b<\a$, $\lusim{p}_1(\b)\sse\MPB_U\times\om$ is a
$\MPB_U-$name for a natural number such

$\hspace{.65cm}$ that
\begin{itemize}
\item  $\lusim{p}_1(\b)$ is partial function from $\MPB_U$ into
$\om$, \item for some  fixed $l<\om$, $\dom
\lusim{p}_1(\b)\sse\{\lan s,\om \setminus \max(s)+1\ran:
s\in[\om]^l\}$, \item for all $\b_1\neq\b_2<\a, \range(\lusim{p}_1(\b_1))\cap \range(\lusim{p}_1(\b_2))$ is finite \footnote{Thus if $G$ and $r$ are as in the proof of Lemma 3.3 with $p_0\in G,$ then $p_o\vdash ``\lusim{p}_1(\b) $ is the $l$-th element of $r$''.}.
\end{itemize}

(2b) for every  $I\sse\a$, $I\in V$, $p'_0\leq p_0$ and
finite $J\sse\om$ there is a finite set

$\hspace{.8cm}$$a\sse\a$ such that for every finite set $b\sse
I\setminus a$ there is $p''_0\leq^* p'_0$ such

$\hspace{.8cm}$that $p''_0\vdash  (\forall \b\in b, \forall
k\in J, \lusim{p}_1(\b)\neq k) \& (\forall \b_1\neq\b_2\in b,
\lusim{p}_1(\b_1)\neq \lusim{p}_1(\b_2))$.
\end{definition}

\begin{notation} (1) Call $\a$ the length of $p$ (or $\lusim{p}_1$) and denote
it by $\len(p)$ (or $\len(\lusim{p}_1))$.

(2) For $n<\om$ let $\lusim{I}_{p,n}$ be a $\MPB_U$-name such that
$p_0\vdash  \lusim{I}_{p,n}=\{\b<\a:
\lusim{p}_1(\b)=n\}$. Then we can identify
$\lusim{p}_1$ with $\lan \lusim{I}_{p,n}:n<\om\ran$.
\end{notation}
\begin{remark} (2a) will guarantee that for $\b<\a$, $p_0\vdash
\lusim{p}_1(\b)\in\om $. The last condition in (2a) is
a technical fact that will be used in several parts of the
argument. The condition (2b) appears technical but it will be
crucial for producing numerous Cohen reals.
\end{remark}
\begin{definition} For $p=\lan p_0,\lusim{p}_1\ran, q=\lan
q_0,\lusim{q}_1\ran\in\MPB$, define

$(1)$ $p\leq q$ iff
\begin{itemize}
\item $p_0\leq_{\MPB_U} q_0$, \item $\len(q)\leq \len(p)$,

\item $p_0\vdash  \forall n<\om,
\lusim{I}_{q,n}=\lusim{I}_{p,n}\cap \len(q)$.
\end{itemize}

$(2)$ $p\leq^* q$ iff \begin{itemize} \item $p_0\leq^*_{\MPB_U}
q_0 $, \item $p\leq q$. \end{itemize}

we call $\leq^*$ a direct or $*$-extension.
\end{definition}
\begin{remark} In the definition of $p\leq q$, we can replace the last condition by
$p_0\vdash   \lusim{q}_1=\lusim{p}_1\upharpoonright \len(q)
$.
\end{remark}
\begin{lemma} Let $\lan p_0,\lusim{p}_1\ran\vdash `` \a$ is an
ordinal''. Then there are $\MPB_U$-names $\lusim{\b}$
and $\lusim{q}_1$ such that $\lan p_0,\lusim{q}_1\ran\leq^*\lan
p_0,\lusim{p}_1\ran$ and $\lan p_0,\lusim{q}_1\ran\vdash
\lusim{\a}=\lusim{\b}$.
\end{lemma}
\begin{proof} Suppose for simplicity that $\lan
p_0,\lusim{p}_1\ran=\lan\lan<>,\om\ran,\phi\ran$. Let $\theta$ be
large enough regular and let $\lan N_n:n<\om\ran$ be an increasing
sequence of countable elementary submodels of $H_\theta$ such that
 $\MPB$, $\lusim{\a}\in N_0$ and $N_n\in N_{n+1}$ for each $n<\om$.
Let $N=\ds\bigcup_{n<\om} N_n$, $\d_n=N_n\cap\om_1$ for $n<\om$
and $\d=\ds\bigcup_{n<\om}\d_n=N\cap \om_1$. Let $\lan
J_n:n<\om\ran\in N_0$ be a sequence of infinite subsets of
$\om\setminus\{0\}$ such that $\ds\bigcup_{n<\om}
J_n=\om\setminus\{0\},$ $J_n\sse J_{n+1},$ and $J_{n+1}\setminus
J_n$ is infinite for each $n<\om$.
 Also let
$\lan\a_i:0<i<\om\ran$ be an enumeration of $\d$ such that
 for
every $n<\om$, $\{\a_i:i\in J_n\}\in N_{n+1}$ is an enumeration of
$\d_n$ and $\{\a_i:i\in J_{n+1}\}\cap\d_n=\{\a_i:i\in J_n\}$.

~~~~~~~~~~~~~~~~~~~~~~~~~

We define by induction on the length of $s,$ a sequence $\lan p^s:s\in[\om]^{<\om}\ran$
of conditions such that

\begin{itemize}
\item $p^s=\lan p_0^s,\lusim{p}_1^s\ran=\lan \lan s,A_s \ran,
\lusim{p}_1^s \ran$, \item $p^s\in N_{s(\len(s)-1)+1}$, \item
$\len(p^s)=\d_{s(\len(s)-1)+1}$, \item if $t$ does not contradict
$p_0^s$ (i.e if $t$ end extends $s$ and $t\setminus s\sse A_S$)
then $p^t\leq p^s$.
\end{itemize}

For $s=<>$, let $p^{<>}=\lan\lan <>,\om\ran,\phi\ran$. Suppose
that $<>\neq s\in[\om]^{<\om}$ and                     $p^{s \upharpoonright
\len(s)-1}$ is defined. We define $p^s$. First we define $t^{s \upharpoonright
\len(s)-1}\leq^* p^{s \upharpoonright \len(s)-1}$ as follows: If there is no
$*$-extension of $p^{s \upharpoonright \len(s)-1}$ deciding $\lusim{\a}$ then
let $t^{s \upharpoonright \len(s)-1}=p^{s \upharpoonright \len(s)-1}$. Otherwise let $t^{s
\upharpoonright \len(s)-1}\in N_{s(\len(s)-2)+1}$ be such an extension. Note that
$\len(t^{s \upharpoonright \len(s)-1} ) \leq \d_{s(\len(s)-2)+1}$.

~~~~~~~~~~~~~~~~~~~~~~~~~~

Let $t^{s \upharpoonright \len(s)-1}=\lan t_0,\lusim{t}_1\ran, t_0=\lan s\upharpoonright \len(s)-1,A\ran$. Let $C\sse\om$ be an infinite set almost  disjoint to
$\lan \range(\lusim{t}_1(\b)):\b < \len (\lusim{t}_1)\ran$. Split $C$
into $\om$ infinite disjoint sets $C_i$, $i<\om$. Let $\lan
c_{ij}:j<\om\ran$ be an increasing enumeration of $C_i$, $i<\om$.
We may suppose that all of these is done in $N_{s(\len(s)-1)+1}$. Let
$p^s=\lan p_0^s,\lusim{p}_1^s\ran$, where

\begin{itemize}
\item $p_0^s=\lan s,A\setminus (\max(s)+1)\ran$, \item for
$\b<\len(\lusim{t}_1)$, $\lusim{p}_1^s(\b)=\lusim{t}_1(\b),$ \item
for $i\in J_{s(\len(s)-1)}$  such that $\a_i\in\d_{s(\len(s)-1)}\setminus \len(\lusim{t}_1)$
\end{itemize}
\begin{center}
$\lusim{p}_1^s(\a_i)=\left\{\lan\lan s^{\frown}\lan
r_1,...,r_i\ran,\om\setminus (r_i+1)\ran , c_{ir_i}\ran:r_1> \max(s),\lan r_1,...,r_i\ran\in [\om]^i\right\}.$
\end{center}

Trivially $p^s\in N_{s(\len(s)-1)+1}$, $\len(p^s)=\d_{s(\len(s)-1)}$, and
if $s(\len(s)-1)\in A$, then $p^s\leq t^{s \upharpoonright \len(s)-1}$.

\begin{claim} $p^s\in\MPB$. \end{claim}

\begin{proof} We check conditions in Definition 3.4.

~~~~~~~~~~~~~~~~~~~~~~

(1) i.e. $p_0^s\in\MPB_U$ is  trivial.

~~~~~~~~~~~~~~~~~~~~~~~

(2) It is clear that $p_0^s\vdash   \lusim{p}_1^s:
\d_{s(\len(s)-1)}\lora\om $ and that (2a) holds. Let us
prove (2b). Thus suppose that $I\sse \d_{s(\len(s)-1)}$, $I\in V$,
$p\leq p_0^s$ and $J\sse\om$ is finite. First we apply (2b) to
$\lan p, \lusim{t}_1\ran , I\cap \len(\lusim{t}_1)$, $p$ and $J$ to
find a finite set $a'\sse \len(\lusim{t}_1)$ such that

~~~~~~~~~~~~~~~~~~~~~~~~~~~~~~~~~~~~~~~~~~~~~~~~

$(*)$$\hspace{1cm}$ For every finite set $b\sse I\cap
\len(\lusim{t}_1)\setminus a'$ there is $p'\leq^* p$ such that

$\hspace{1.5cm}$$p' \vdash  (\forall \b\in b, \forall k\in
J, \lusim{t}_1(\b)\neq k) \& (\forall \b_1\neq\b_2\in b,
\lusim{t}_1(\b_1)\neq \lusim{t}_1(\b_2))$.

~~~~~~~~~~~~~~~~~~~~~~~~~~~~~~~~~~~~~~~~~~~~~~~~~~~~

Let $p=\lan s^{\frown}\lan r_1,...,r_m\ran,B\ran$. Suppose that
$\d_{s(\len(s)-1)}\setminus
\len(\lusim{t}_1)=\{\a_{J_1},...,\a_{J_i},...\}$ where $J_1<J_2<...$
are in $J_{s(\len(s)-1)}$. Let
\begin{center}
$a=a'\cup \{\a_{J_1},...,\a_{J_m}\}$.
\end{center}
 We show that $a$ is as required.
Thus suppose that $b\sse I\setminus a$ is finite. Apply $(*)$ to
$b\cap \len(\lusim{t}_1)$ to find $p'=\lan s^{\frown}\lan
r_1,...,r_m\ran,B'\ran\leq^* p$ such that

\begin{center}
 $p' \vdash
(\forall \b\in b\cap \len(\lusim{t}_1),  \forall k\in J,
\lusim{t}_1(\b)\neq k) \& (\forall \b_1\neq\b_2\in b \cap
\len(\lusim{t}_1), \lusim{t}_1(\b_1)\neq \lusim{t}_1(\b_2))
$.
\end{center}

Also note that

\begin{center}
$p' \vdash  \forall \b\in b\cap \len(\lusim{t}_1),
\lusim{p}_1^s(\b)=\lusim{t}_1(\b)$.
\end{center}

 Pick
$k<\om$ such that
\[\forall \b\in b\cap \len(\lusim{t}_1),  \forall \a_i\in b\setminus \len(\lusim{t}_1), \range(\lusim{p}_1^s(\b_1)) \cap (\range(\lusim{p}_1^s(\a_i))\setminus k)=\phi.\]

Let $q=\lan s^{\frown}\lan r_1,...,r_m\ran,B\ran=\lan s^{\frown}\lan
r_1,...,r_m\ran,B'\setminus(\max(J)+k+1)\ran$. Then $q \leq^*
p'\leq^* p$. We show that $q$ is as required. We need to show
that

\begin{enumerate}
\item $q \vdash   \forall \b\in b\setminus
\len(\lusim{t}_1), \forall k\in J, \lusim{p}_1^s(\b)\neq k
$, \item  $q \vdash  \forall \b_1\neq\b_2 \in b\setminus
\len(\lusim{t}_1), \lusim{p}_1^s(\b_1)\neq \lusim{p}_1^s(\b_2)
$, \item $q \vdash  \forall \b_1\in b\cap
\len(\lusim{t}_1), \forall \b_2\in b\setminus \len(\lusim{t}_1),
\lusim{p}_1^s(\b_1)\neq \lusim{p}_1^s(\b_2)$.
\end{enumerate}

Now (1) follows from the fact that $q \vdash ``
\lusim{p}_1^s(\a_i)\geq (i-m)$-th element of $B>\max(J)$''.
(2) follows from the fact that for $i\neq j<\om$, $C_i\cap
C_j=\emptyset$, and $\range(\lusim{p}_1^s(\a_i))\sse C_i$. (3) follows
from the choice of $k$. The claim follows.
\end{proof}

This completes our definition of the sequence $\lan
p^s:s\in[\om]^{<\om}\ran$. Let

\begin{center}
$ \lusim{q}_1=\{\lan p_0^s,\lan\b, \lusim{p}_1^s(\b)\ran\ran :
s\in[\om]^{<\om}, \b< \len(p^s)\}$.
\end{center}

 Then $
\lusim{q}_1$ is a $\MPB_U$-name and for $s\in[\om]^{<\om}$,
$p_0^s\vdash   \lusim{p}_1^s=\lusim{q}_1 \upharpoonright
\len(\lusim{p}_1^s)$.

\begin{claim} $\lan\lan <>,\om\ran,\lusim{q}_1\ran \in\MPB$.
\end{claim}

\begin{proof} We check conditions in Definition 3.4.

~~~~~~~~~~

 (1) i.e.
$\lan <>,\om\ran \in\MPB_U$ is  trivial.

~~~~~~~~~~~~~

 (2) It is clear from our
definition that

\begin{center}
 $\lan <>,\om\ran\vdash `` \lusim{q}_1$ is a
well-defined function into $\om$''.
\end{center}

Let us show that $\len(\lusim{q}_1)=\d$. By the construction it is
trivial that $\len(\lusim{q}_1)\leq\d$. We show that
$\len(\lusim{q}_1)\geq\d$. It suffices  to prove the following

\begin{center}
$(*)$~~ For every $\tau<\d$ and $p\in\MPB_U$ there is $q\leq p$
such that $q \vdash `` \lusim{q}_1(\tau)$ is defined
''.
\end{center}
Fix $\tau<\d$  and $p=\lan s,A\ran\in\MPB_U$ as in $(*)$. Let $t$ be
an end extension of $s$ such that $t\setminus s\sse A$ and
$\d_{t(\len(t)-1)}>\tau$. Then $p_0^t$ and $p$ are compatible and
$p_0^t \vdash ``  \lusim{q}_1(\tau)= \lusim{p}_1^t(\tau) $
is defined''. Let $q\leq p_0^t, p$. Then $q \vdash
`` \lusim{q}_1(\tau)$ is defined'' and $(*)$
follows. Thus $\len(\lusim{q}_1)=\d$.

~~~~~~

(2a) is trivial. Let us prove (2b). Thus suppose that $I\sse\d$,
$I\in V$, $p\leq \lan <>,\om\ran$ and $J\sse\om$ is finite. Let
$p=\lan s,A\ran$.

 First
we consider the case where $s=<>$. Let $a=\emptyset$. We show that
$a$ is as required. Thus let $b\sse I$ be finite. Let $n\in A$ be
such that $n>\max(J)+1$ and $b\sse\d_n$. Let $t=s^{\frown}\lan n\ran$.
Note that

\begin{center}
$\forall \b_1\neq\b_2\in b$, $\range(\lusim{p}_1^t(\b_1))\cap \range(\lusim{p}_1^t(\b_2))=\emptyset$.
\end{center}

 Let $q=\lan <>,B\ran=\lan
<>,A\setminus(\max(J)+1)\ran$. Then $q\leq^* p$ and $q$ is
compatible with $p_0^t$. We show that $q$ is as required. We need
to show that
\begin{enumerate}
\item $q \vdash   \forall \b\in b, \forall k\in J,
\lusim{q}_1(\b)\neq k $, \item  $q \vdash
\forall \b_1\neq\b_2 \in b, \lusim{q_1}(\b_1)\neq
\lusim{q}_1(\b_2)$.
\end{enumerate}

For (1), if it fails, then we can find $\lan r,D\ran\leq q,p_0^t$,
$\b\in b$ and $k\in J$ such that $\lan r,D\ran\leq^* p_0^r$ and
$\lan r,D\ran \vdash   \lusim{q}_1(\b)=k$. But
$\lan r,D\ran \vdash  \lusim{q}_1(\b)=
\lusim{p}_1^r(\b)= \lusim{p}_1^t(\b)$, hence $\lan
r,D\ran \vdash  \lusim{p}_1^t(\b)=k $. This is
impossible since $\min(D)\geq \min(B)>\max(J)$. For (2), if it fails,
then we can find $\lan r,D\ran\leq q, p_0^t$ and $\b_1\neq\b_2\in
b$ such that $\lan r,D\ran\leq^* p_0^r$ and $\lan r,D\ran \vdash
 \lusim{q}_1(\b_1)= \lusim{q}_1(\b_2) $. As
above it follows that $\lan r,D\ran \vdash
\lusim{p}_1^t(\b_1)= \lusim{p}_1^t(\b_2)$. This is
impossible since for $\b_1\neq\b_2\in b$, $\range(\lusim{p}_1^t(\b_1))\cap \range(\lusim{p}_1^t(\b_2))=\emptyset$. Hence
$q$ is as required and we are done.

~~~~~~

Now consider the case $s\neq<>$. First we apply (2b) to $t^s$,
$I\cap \len(t^s)$, $p$ and $J$ to find a finite set $a'\sse \len(t^s)$
such that

~~~~~~~~~~

$(**)$ $\hspace{1cm}$ For every finite set $b\sse I\cap
\len(t^s)\setminus a'$ there is $p'\leq^* p$ such that $p'$

$\hspace{1.8cm}$$\vdash  (\forall \b\in b, \forall k\in
J$, $\lusim{p}_1^s(\b)\neq k) \& (\forall \b_1\neq  \b_2\in b,
\lusim{p}_1^s(\b_1)\neq \lusim{p}_1^s(\b_2))$

~~~~~~

Let $t^s=\lan t_0,\lusim{t}_1\ran, \d_{s(\len(s)-1)+1}\setminus
\d_{s(\len(s)-1)}=\{\a_{J_1},\a_{J_2},...\}$, where $J_1<J_2<...$ are
in $J_{s(\len(s)-1)+1}$. Define

\begin{center}
$a=a'\cup\{\a_1,\a_2,...,\a_{J_{\len(s)+1}}\}$.
\end{center}

 We show that $a$ is
as required. First apply $(**)$ to $b\cap \len(t^s)$ to find $p'=\lan
s,A'\ran\leq^* p$ such that

\begin{center}
$p' \vdash   (\forall \b\in b\cap \len(t^s), \forall k\in
J,  \lusim{t}_1(\b)\neq k) \& (\forall \b_1\neq\b_2\in b\cap
\len(t^s), \lusim{t}_1(\b_1)\neq \lusim{t}_1(\b_2)) $.
\end{center}

 Pick $n\in A'$ such that $n>\max(J)+1$
and $b\sse\d_n$ and let $r=s^{\frown}\lan n\ran$. Then
\[\forall \b_1\neq\b_2\in b\setminus \len(t^s), \range(\lusim{p}_1^r(\b_1))\cap \range(\lusim{p}_1^r(\b_2))=\emptyset.\]
Pick $k<\om$ such that $k>n$ and
\[\forall \b_1\in b\cap \len(t^s),\forall \b_2\in b\setminus \len(t^s), \range(\lusim{p}_1^r(\b_1))\cap (\range(\lusim{p}_1^r(\b_2))\setminus k)=\emptyset.\]

Let $q=\lan s,B\ran=\lan s,A'\setminus(\max(J)+k+1)\cup\{n\}\ran$.
Then $q\leq^*p'\leq^* p$ and $q$ is compatible with $p_0^r$ (since
$n\in B$).
 We show that $q$ is as required. We need to prove the following
\begin{enumerate}
\item $q \vdash   \forall \b\in b,  \forall k\in J,
\lusim{q}_1(\b)\neq k $, \item  $q \vdash
\forall \b_1\neq\b_2 \in b\setminus \len(t^s),
\lusim{q}_1(\b_1)\neq \lusim{q}_1(\b_2)$, \item  $q
\vdash  \forall \b_1\in b\cap \len(t^s),  \forall \b_2\in
b\setminus \len(t^s), \lusim{q}_1(\b_1)\neq \lusim{q}_1(\b_2)
$.
\end{enumerate}
The proofs of (1) and (2) are as in the case $s=<>$. Let us prove
(3). Suppose that (3) fails. Thus we can find $\lan u,D\ran\leq
q,p_0^r$, $\b_1\in b\cap \len(t^s)$ and $\b_2\in b\setminus
\len(t^s)$ such that $\lan u,D\ran \leq^* p_0^u$ and $\lan u,D\ran
\vdash   \lusim{q}_1(\b_1)=\lusim{q}_1(\b_2)$.
But for $\b\in b$, $\lan u,D\ran \vdash
\lusim{q}_1(\b)=\lusim{p}_1^u(\b)= \lusim{p}_1^r(\b),$
 hence $\lan u,D\ran \vdash
\lusim{p}_1^r(\b_1)=\lusim{p}_1^r(\b_2)$. Now note that
$\b_2=\a_i$ for some $i>\len(s)+1$, $\min(D)\geq n$ and
$\min(D\setminus\{n\})>k$, hence by the construction of $p^r$

\begin{center}
$\lan u,D\ran \vdash ``  \lusim{p}_1^r(\b_2)\geq (i-\len(s))$-th element of $D>k $''.
\end{center}

 By our choice of $k$, $\range(\lusim{p}_1^r(\b_1))\cap (\range(\lusim{p}_1^r(\b_2))\setminus
k)=\emptyset $ and we get a contradiction. (3) follows. Thus $q$
is as required, and the claim follows.
\end{proof}

 Let

\begin{center}
$\lusim{\b}=\{\lan p_0^s,\d\ran: s\in[\om]^{<\om}, \exists \gamma
(\d<\gamma, p^s \vdash  \lusim{\a}=\gamma$$ )\}$.
\end{center}

Then $\lusim{\b}$ is a $\MPB_U$-name of an ordinal.

\begin{claim} $\lan\lan <>,\om\ran,\lusim{q}_1\ran \vdash
\lusim{\a}=\lusim{\b}$.
\end{claim}

\begin{proof} Suppose not. There are two cases to be considered.

~~~

 {\bf Case 1.} There are $\lan r_0, \lusim{r}_1\ran\leq \lan\lan <>,\om\ran,\lusim{q}_1\ran$ and $\d$ such that
$\lan r_0, \lusim{r}_1\ran \vdash ``  \d\in \lusim{\a}$ and
$\d\not\in \lusim{\b} $''. We may suppose that for some
ordinal $\a$, $\lan r_0, \lusim{r}_1\ran \vdash
\lusim{\a}=\a $. Then $\d<\a$. Let $r_0=\lan s,A\ran$.
Consider $p^s=\lan p_0^s, \lusim{p}_1^s\ran$. Then $p_0^s$ is
compatible with $r_0$ and there is a $*$-extension of $p^s$
deciding $\lusim{\a}$. Let $t\in N_{s(\len(s)-1)+1}$ be the
$*-$extension of $p^s$ deciding $\lusim{\a}$ chosen in the proof
of Lemma 3.9. Let $t=\lan t_0, \lusim{t}_1\ran, t_0=\lan s,B\ran$,
and let $\g$ be such that $\lan t_0, \lusim{t}_1\ran \vdash
 \lusim{\a}=\g$. Let $n\in A\cap B$. Then
\begin{itemize}
\item $p_0^{s^{\frown}\lan n\ran},t_0$ and $p_0^s$ are compatible and
$\lan s^{\frown}\lan n\ran, A\cap B\cap A_{s^{\frown}\lan n\ran}\ran$
extends them, \item $p^{s^{\frown}\lan n\ran}\leq t$.
\end{itemize}
Thus $p^{s^{\frown}\lan n\ran}\vdash  \lusim{\a}=\g
$. Let $u=\lan s^{\frown}\lan n\ran, A\cap B\cap A_{s^{\frown}\lan
n\ran}\setminus (n+1)\ran$.

Then $u\leq p_0^{s^{\frown}\lan n\ran}$ and $u\vdash ``
\lusim{r}_1$ extends $\lusim{p}_1^{s^{\frown}\lan n\ran}$ which
extends $\lusim{t}_1$''. Thus $\lan u, \lusim{r}_1\ran\leq
t, \lan r_0, \lusim{r}_1\ran, p^{s^{\frown}\lan n\ran}$. It follows
that $\a=\g$. Now $\d<\g$ and $p^{s^{\frown}\lan n\ran}\vdash
 \lusim{\a}=\g$. Hence $\lan p_0^{s^{\frown}\lan
n\ran},\d\ran\in \lusim{\b}$ and $p^{s^{\frown}\lan n\ran} \vdash
  \d\in \lusim{\b}$. This is impossible since
$\lan r_0, \lusim{r}_1\ran\vdash  \d \not\in
\lusim{\b}$.

~~~~~

{\bf Case 2.} There are $\lan r_0, \lusim{r}_1\ran\leq \lan\lan
<>,\om\ran,\lusim{q}_1\ran$ and $\d$ such that $\lan r_0,
\lusim{r}_1\ran\vdash ``  \d\in \lusim{\b}$ and $\d\not\in
\lusim{\a} $''. We may further suppose that for some
ordinal $\a$, $\lan r_0, \lusim{r}_1\ran\vdash
\lusim{\a}=\a $. Thus $\d\geq\a$. Let $r=\lan s,A\ran$.
Then as above $p_0^s$ is compatible with $r$ and there is a
$*$-extension of $p^s$ deciding $\lusim{\a}$. Choose $t$ as in
Case 1, $t=\lan t_0, \lusim{t}_1\ran, t_0=\lan s,B\ran$ and let
$\g$ be such that $\lan t_0, \lusim{t}_1\ran\vdash
\lusim{\a}=\g $. Let $n\in A\cap B$. Then as in Case 1,
$\a=\g$ and $p^{s^{\frown}\lan n\ran}\vdash  \lusim{\a}=\g
$. On the other hand since $\lan r_0,
\lusim{r}_1\ran\vdash  \d\in \lusim{\b}$, we
can find $\ov{s}$ such that $\ov{s}$ does not contradict
$p_0^{s^{\frown}\lan n\ran}$, $\lan p_0^{\ov{s}}, p_1^{\ov{s}}\ran
\vdash  \lusim{\a}= \ov{\g}, $ for some
$\ov{\g}>\d$ and $\lan p_0^{\ov{s}},\d\ran\in \lusim{\b}$. Now
$\ov{\g}=\g=\a>\d$ which is in contradiction with $\d\geq\a$. The
claim follows.
\end{proof}

This completes the proof of Lemma 3.9.
\end{proof}

\begin{lemma} Let $\lan p_0, \lusim{p}_1\ran\vdash  \lusim{f}: \om
\lora ON $. Then there are
 $\MPB_U$-names $\lusim{g}$ and $\lusim{q}_1$ such that $\lan p_0, \lusim{q}_1\ran \leq^* \lan p_0, \lusim{p}_1\ran$
 and $\lan p_0, \lusim{q}_1\ran\vdash   \lusim{f}= \lusim{g} $.
\end{lemma}
 \begin{proof} For simplicity suppose that $\lan p_0, \lusim{p}_1\ran=\lan\lan<>,\om\ran,\emptyset\ran$.
 Let $\theta$ be large enough regular and let $\lan N_n:n<\om\ran$ be an increasing sequence of
 countable elementary submodels of $H_\theta$ such that $\MPB, \lusim{f}\in N_0$ and $N_n\in N_{n+1}$ for every $n<\om$.
 Let $N=\ds\bigcup_{n<\om} N_n$, $\d_n=N_n\cap\om_1$ for $n<\om$ and $\d=\ds\bigcup_{n<\om} \d_n=N\cap\om_1$.
 Let $\lan J_n:n<\om\ran\in N_0$ and
 $\lan\a_i: 0<i<\om\ran$ be as in Lemma 3.9.

 We define by induction a sequence $\lan p^s: s\in[\om]^{<\om}\ran$ of conditions and a
 sequence $\lan \lusim{\b}_s:s\in[\om]^{<\om}\ran$ of $\MPB_U$-names for ordinals such that

\begin{itemize}
\item $p^s=\lan p_0^{s},\lusim{p}_1^s\ran=\lan \lan s,\om
\setminus(\max(s)+1)\ran, \lusim{p}_1^s\ran$, \item $p^s\in
N_{s(\len(s)-1)+1}$, \item $\len(p^s)\geq \d_{s(\len(s)-1)}$, \item $p^s
\vdash ``  \lusim{f}(\len(s)-1)=\lusim{\b}_s$'', \item
if $t$ end extends $s$, then $p^t\leq p^s$.
\end{itemize}

 For $s=<>$, let $p^{<>}=\lan\lan<>,\om\ran,\emptyset\ran$. Now suppose that $s\neq <>$ and
 $p^{s\upharpoonright \len(s)-1}$ is defined. We define $p^s$. Let $C_{s\upharpoonright \len(s)-1}$ be an infinite subset of $\om$ almost
 disjoint to $\lan \range(\lusim{p}_1^{s \upharpoonright \len(s)-1}(\b)): \b < \len(p^{s \upharpoonright \len(s)-1})\ran$.
 Split $C_{s \upharpoonright \len(s)-1}$ into $\om$ infinite disjoint sets $\lan C_{s \upharpoonright \len(s)-1,t}: t\in[\om]^{<\om}$ and $t$ end
 extends $s \upharpoonright \len(s)-1\ran$. Again split $C_{s \upharpoonright \len(s)-1,s}$ into $\om$ infinite disjoint
 sets $\lan C_i:i<\om\ran$. Let $\lan c_{ij}: j<\om\ran$ be an
  increasing enumeration of $C_i$, $i<\om$. We may suppose that all of these is done
  in $N_{s(\len(s)-1)+1}$. Let $q^s=\lan q_0^s, \lusim{q}_1^s\ran$, where

\begin{itemize}
 \item $q_0^s=\lan s,\om\setminus(\max(s)+1)\ran$,
 \item for $\b < \len(p^{s \upharpoonright \len(s)-1}), \lusim{q}_1^s(\b)=\lusim{p}_1^{s \upharpoonright \len(s)-1}(\b)$, \item for $i \in J_{s(\len(s)-1)}$
 such that $\a_i \in \d_{s(\len(s)-1)}\setminus \len(p^{s \upharpoonright \len(s)-1})$
\begin{center}
$ \lusim{q}_1^s(\a_i)=\{ \lan\lan s^{\frown}\lan r_1,...,r_i\ran,
\om\setminus (r_i+1)\ran, c_{ir_i}\ran: r_1>\max(s), \lan
r_1,...,r_i\ran\in [\om]^i\}.$
\end{center}
\end{itemize}

 Then $q^s\in N_{s(\len(s)-1)+1}$ and as in the proof of claim 3.10, $q^s\in\MPB$. By Lemma 3.9,
 applied inside $N_{s(\len(s)-1)+1},$ we can find $\MPB_U$-names $\lusim{\b}_s$ and $\lusim{p}_1^s$ such
 that $\lan q_0^s, \lusim{p}_1^s\ran\leq \lan q_0^s, \lusim{q}_1^s\ran$
 and $\lan q_0^s, \lusim{p}_1^s\ran\vdash   \lusim{f}(\len(s)-1)=\lusim{\b}_s$.
 Let $p^s=\lan p_0^s, \lusim{p}_1^s\ran=\lan q_0^s, \lusim{p}_1^s\ran$. Then $p^s\leq p^{s \upharpoonright \len(s)-1}$ and $p^s \vdash
 \lusim{f} \upharpoonright \len(s)=\{ \lan i, \lusim{\b}_{s \upharpoonright i+1}\ran : i<\len(s)\}$.

 This completes our definition of the sequences $\lan p^s : s\in[\om]^{<\om}\ran$ and $\lan \lusim{\b}_s : s\in[\om]^{<\om}\ran$. Let
 \begin{eqnarray*}
  \lusim{q}_1 &=& \{ \lan p_0^s, \lan \b, \lusim{p}_1^s(\b)\ran\ran: s\in[\om]^{<\om}, \b< \len(p^s)\},\\
  \lusim{g} &=& \{\lan p_0^s, \lan i,\lusim{\b}_{s \upharpoonright i+1}\ran\ran: s\in[\om]^{<\om}, i< \len(s)\}.
 \end{eqnarray*}

 Then $\lusim{q}_1$ and $\lusim{g}$ are $\MPB_U$-names.

\begin{claim} $\lan\lan <>,\om \ran,\lusim{q}_1\ran\in\MPB$.
\end{claim}

 \begin{proof} We check conditions in Definition 3.4.

 ~~~~~~

(1) i.e $\lan <>,\om\ran\in\MPB_U$ is trivial.

~~~~~~

 (2) It is clear by our construction  that
 \begin{center}
$\lan <>,\om\ran \vdash ``  \lusim{q}_1$ is a well-defined
function''
\end{center}
and as in the proof of claim 3.11, we can show that
$\len(\lusim{q}_1)=\d$. (2a) is trivial. Let us prove (2b). Thus
suppose that $I\sse\d$, $I\in V$, $p\leq\lan <>,\om\ran$ and
$J\sse\om$ is finite. Let $p=\lan s,A\ran$. If $s=<>$, then as in
the proof of 3.11, we can show that $a=\emptyset$ is a required.
Thus suppose that $s\neq <>$. First we apply (2b) to $p^s$, $I\cap
\len(p^s)$, $p$ and $J$ to find $a'\sse \len(p^s)$ such that

~~~~~~

$(*)$ $\hspace{1cm}$ For every finite $b\sse I\cap
\len(p^s)\setminus a'$ there is $p'\leq^* p$ such that $p'$

$\hspace{1.6cm}$$ \vdash   (\forall \b\in b, \forall k\in
J, \lusim{p}_1^s(\b)\neq k) \& (\forall \b_1\neq\b_2\in b,
\lusim{p}_1^s(\b_1)\neq \lusim{p}_1^s(\b_2))$.

~~~~~

Let $\d_{s(\len(s)-1)+1}\setminus \d_{s(\len(s)-1)}=\{\a_{J_1},...,\a_{J_i},...\}$ where $J_1<J_2<...$ are in
$J_{s(\len(s)-1)+1}$. Let

\begin{center}
 $a=a'\cup \{\a_{1},\a_2,...,\a_{J_{\len(s)}}\}$.
\end{center}

 We  show that $a$ is as required. Let $b\sse I\setminus a$ be finite. First we apply $(*)$ to $b\cap \len(p^s)$
to find $p'=\lan s,A'\ran\leq^* p$ such that

\begin{center}
$p' \vdash  (\forall \b\in b \cap \len(p^s), \forall k\in
J, \lusim{p}_1^s(\b)\neq k) \& (\forall \b_1\neq\b_2\in b\cap
\len(p^s), \lusim{p}_1^s(\b_1)\neq \lusim{p}_1^s(\b_2))$.
\end{center}

 Also note
that for $\b\in b \cap \len(p^s)$, $p' \vdash
\lusim{q}_1(\b)=\lusim{p}_1^s(\b))$. Pick $m$ such that
$\max(s)+\max(J)+1<m<\om$ and if $t$ end extends $s$ and $m<\max(t)$,
then $C_{s,t}$ is disjoint to $J$ and to $\range(\lusim{p}_1^s(\b))$
for $\b\in b \cap \len(p^s)$. Then pick $n>m, n\in A'$ such that
$b\sse\d_n$, and let $t=s^{\frown}\lan n\ran$. Then
\begin{itemize}
\item $\forall\b_1\neq\b_2\in b\setminus \len(p^s)$, $\range(\lusim{p}_1^t(\b_1))\cap \range(\lusim{p}_1^t(\b_2))=\emptyset$, \item
$\forall\b_1\in b\cap \len(p^s),\forall\b_2\in b\setminus \len(p^s)$,
$\range(\lusim{p}_1^t(\b_1))\cap \range(\lusim{p}_1^t(\b_2))=\emptyset$,
\item $\forall\b\in b\setminus \len(p^s), \range(\lusim{p}_1^t(\b))\cap
J=\emptyset$.
\end{itemize}

Let $q=\lan s,B\ran=\lan s,A'\setminus (n+1)\ran$. Then $q\leq^*
p'\leq^* p$ and using the above facts we can show that

\begin{center}
 $q \vdash  (\forall \b\in b, \forall k\in J, \lusim{q}_1(\b)=\lusim{p}_1^t(\b)\neq k) \& (\forall \b_1\neq\b_2\in b,  \lusim{q}_1(\b_1)=\lusim{p}_1^t(\b_1)\neq \lusim{p}_1^t(\b_2)=\lusim{q}_1(\b_2))$.
\end{center}

 Thus $q$ is as required and the claim follows.
 \end{proof}

\begin{claim} $\lan\lan <>,\om \ran,\lusim{q}_1\ran \vdash
\lusim{f}=\lusim{g}$.
\end{claim}

 \begin{proof} Suppose not. Then we can find $\lan r_0, \lusim{r}_1\ran\leq \lan\lan <>,\om\ran,\lusim{q}_1\ran$
 and $i<\om$ such that $\lan r_0, \lusim{r}_1\ran \vdash    \lusim{f}(i)\neq \lusim{g}(i)$. Let $r_0=\lan s,A\ran$.
  Then $r_0$ is compatible with $p_0^s$ and $r_0\vdash ``  \lusim{r}_1$ extends $p_1^s$''.
  Hence $\lan r_0, \lusim{r}_1\ran\leq \lan
 p_0^s,\lusim{p}_1^s\ran=p^s$. Now $p^s \vdash   \lusim{g}(i)=\lusim{\b}_{s \upharpoonright i+1}=\lusim{f}(i),$ and
 we get a contradiction. The claim follows.
 \end{proof}

 This completes the proof of Lemma 3.13.
 \end{proof}

 The following is now immediate.

\begin{lemma} The forcing $(\MPB,\leq)$ preserves cofinalities.
\end{lemma}

 \begin{proof} By Lemma 3.13, $\MPB$ preserves cofinalities $\leq\om_1$. On the other
  hand by a $\Delta$-system argument, $\MPB$ satisfies the $\om_2$-c.c. and hence it preserves cofinalities $\geq \om_2$.
  \end{proof}

\begin{lemma} Let $G$ be $(\MPB,\leq)$-generic over $V$. Then $V[G]\models
GCH$.
\end{lemma}

 \begin{proof} By Lemma 3.13, $V[G]\models CH$. Now let $\k\geq\om_1$. Then
 \[(2^\k)^{V[G]}\leq ((|\MPB|^{\om_1})^\k)^V\leq (2^\k)^V=\k^+.\]

 The result follows.
 \end{proof}

Now we return to the proof of Theorem 3.1. Suppose that $G$ is
$(\MPB,\leq)$-generic over $V$, and let $V_1=V[G]$. Then $V_1$ is
a cofinality and $GCH$ preserving generic extension of $V$. We
show that adding a Cohen real over $V_1$ produces $\aleph_1$-many
Cohen reals over $V$. Thus force to add a Cohen real over $V_1$.
Split it into $\om$ Cohen reals over $V_1$. Denote them by $\lan
r_{n,m}:n,m<\om\ran$. Also let $\lan f_i:i<\om_1\ran\in V$ be a
sequence of almost disjoint functions from $\om$ into $\om$. First
we define a sequence $\lan s_{n,i} : i<\om_1\ran$ of reals by

\begin{center}
 $\forall k<\om$, $s_{n,i}(k)=r_{n,f_i(k)}(0)$.
\end{center}

Let $\lan I_n : n<\om\ran$ be the partition of $\om_1$ produced by
$G$. For $\a<\om_1$ let

\begin{itemize}
\item $n(\a)=$ that $n<\om$ such that $\a\in I_n$, \item $i(\a)=$
that $i<\om_1$ such that $\a$ is the $i$-th element of
$I_{n(\a)}$.
\end{itemize}

 We define a sequence $\lan t_\a:\a<\om_1\ran$ of reals by $t_\a=s_{n(\a),i(\a)}$. The following
 lemma completes the proof of Theorem 3.1.

\begin{lemma} $\lan t_\a:\a<\om_1\ran$ is a sequence of $\aleph_1$-many
Cohen reals over $V$.
\end{lemma}

 \begin{proof} First note that $\lan r_{n,m}:n,m<\om\ran$ is ${\MCB}(\om\times\om)$-generic
 over $V_1$. By c.c.c. of ${\MCB}(\om_1)$ it suffices to
  show that for every countable $I\sse\om_1$, $I\in V$, $\lan t_\a:\a\in I \ran$ is ${\MCB}(I)$-generic
   over $V$. Thus it suffices to prove the following

~~~~~~~~~~~

$\hspace{1.4cm}$~~~ For every $\lan\lan
p_0,\lusim{p}_1\ran,q\ran\in\MPB*{\MCB}(\om\times\om)$ and every
open dense subset

$(*)$$\hspace{1cm}$ $D\in V$ of ${\MCB}(I)$, there is $\lan\lan
q_0,\lusim{q}_1\ran,r\ran\leq \lan\lan p_0,\lusim{p}_1\ran,q\ran$
such that $\lan\lan q_0,\lusim{q}_1\ran$

$\hspace{1.5cm}$$,r\ran\vdash `` \lan \lusim{t}_\nu :
\nu\in I\ran$ extends some element of $D$''

~~~~~~

Let $\lan\lan p_0,\lusim{p}_1\ran,q\ran$ and $D$ be as above. Let
$\a=\sup(I)$. We may suppose that $\len(\lusim{p}_1)\geq\a$. Let
$J=\{n:\exists m,k,\lan n,m,k\ran\in \dom(q)\}$. We apply (2b) to
$\lan p_0,\lusim{p}_1\ran, I,p_0$ and $J$ to find a finite set
$a\sse I$ such that:

~~~~~~~

$(**)$ $\hspace{1cm}$~ For every finite $b\sse I\setminus a$
there is $p'_0\leq^* p_0$ such that $p'_0 \vdash
(\forall \b$

$\hspace{1.7cm}$$\in b, \forall k\in J, \lusim{p}_1(\b)\neq k) \&
(\forall \b_1 \neq \b_2 \in b, \lusim{p}_1(\b_1)\neq
\lusim{p}_1(\b_2)) $.

~~~~~~

Let
\begin{center}
$S=\{\lan\nu ,k,j\ran: \nu\in a, k<\om, j<2, \lan n(\nu),
f_{i(\nu)}(k),0,j\ran\in q\}$.
\end{center}
Then $S\in {\MCB}(\om_1)$. Pick $k_0<\om$ such that for all
$\nu_1\neq\nu_2\in a$, and $k\geq k_0$, $f_{i(\nu_1)}(k)\neq
f_{i(\nu_2)}(k)$. Let
\begin{center}
 $S^*=S\cup\{\lan \nu,k,0\ran:\nu\in a, k<\k_0$, $\lan\nu,k,1\ran \not\in S\}$.
\end{center}
 The reason for defining $S^*$ is to avoid possible collisions.
 Then $S^*\in {\MCB}(\om_1)$. Pick $S^{**}\in D$ such that $S^{**}\leq S^*$.
 Let $b=\{\nu:\exists k,j, \lan \nu,k,j\ran\in S^{**}\}\setminus q$. By $(**)$
 there is $p'_0\leq^* p_0$ such that

\begin{center}
 $p'_0 \vdash  (\forall
\nu\in b, \forall k\in J, \lusim{p}_1(\nu)\neq k) \& (\forall
\nu_1\neq\nu_2\in b, \lusim{p}_1(\nu_1)\neq  \lusim{p}_1(\nu_2))
$.
\end{center}

 Let $p''_0\leq p'_0$ be such that $\lan p''_0,\lusim{p}_1\ran$ decides all the colors of elements of $a\cup b$. Let
\[q^*=q\cup \{ \lan n(\nu), f_{i(\nu)}(k),0,S^{**}(\nu,k)\ran: \langle \nu,k \rangle\in \dom(S^{**})\}.\]
Then $q^*$ is well defined and $q^*\in {\MCB}(\om\times \om)$. Now
$q^*\leq q$, $\lan\lan p''_0,\lusim{p}_1\ran,q^*\ran\leq \lan\lan
p_0,\lusim{p}_1\ran,q\ran$ and for $\lan\nu,k\ran\in \dom(S^{**})$
\begin{center}
$\lan\lan p''_0,\lusim{p}_1\ran,q^*\ran \vdash
S^{**}(\nu,k)=q^*(n(\nu), f_{i(\nu)}(k),0)=\lusim{r}_{n(\nu),
f_{i(\nu)}(k)}(0)=\lusim{t}_\nu(k)$.
\end{center}

It follows that
\begin{center}
$\lan\lan p''_0,\lusim{p}_1\ran,q^*\ran \vdash`` \lan
\lusim{t}_\nu: \nu\in I\ran$ extends $S^{**}$''.
\end{center}

$(*)$ and hence Lemma 3.18 follows.
\end{proof}

\noindent{\large\bf Acknowledgement}

The authors would like to thank the anonymous referee for carefully reading the paper and offering extensive
improvement to its exposition.

Moti Gitik,
School of Mathematical Sciences, Tel Aviv University, Tel Aviv, Israel.

email: gitik@post.tau.ac.il

Mohammad Golshani,
Kurt G\"odel Research Center for Mathematical Logic (KGRC),

$\hspace{.0cm}$Vienna.

email: golshani.m@gmail.com

\end{document}